\documentclass[a4paper]{amsart}

\usepackage[english]{babel}
\usepackage{flafter}
\usepackage{mathtools}
\usepackage[a4paper,top=3cm,bottom=2cm,left=3cm,right=3cm,marginparwidth=1.75cm]{geometry}

\usepackage{amsmath,amsthm,amssymb,amsfonts}
\usepackage{graphicx}
\usepackage[colorinlistoftodos]{todonotes}
\usepackage[colorlinks=true, allcolors=blue]{hyperref}

\newcommand{\R}{\mathbb{R}}

\newcommand{\CC}{\mathbb{C}}

\newcommand{\Fix}{{\rm Fix}}
\newcommand{\rk}{{\rm rk\ }}

\usepackage{color, colortbl}
\definecolor{Lgray}{rgb}{0.75, 0.75, 0.75}

\theoremstyle{plain}
\newtheorem{theorem}{Theorem}[section]
\newtheorem{proposition}[theorem]{Proposition}

\newtheorem{example}[theorem]{Example}
\newtheorem{remark}[theorem]{Remark}
\newtheorem{lemma}[theorem]{Lemma}
\newtheorem{notation}[theorem]{Notation}
\newtheorem{defin}[theorem]{Definition}
\newtheorem{thmINTRO}{Theorem}

\usepackage{todonotes}

\usepackage{tikz}
\usepackage{pgfplots}
\pgfplotsset{compat=1.15}
\usepackage{mathrsfs}
\usetikzlibrary{arrows}

\usepackage{multirow}
\usepackage{enumerate}
\usepackage{float}

\title{Non-symplectic automorphisms of order multiple of seven on K3 surfaces}

\author[R. Bell]{R. Bell \textsuperscript{1}}
\address{\textsuperscript{1} Department of Mathematics, University of Pennsylvania, 209 South 33rd Street
Philadelphia, PA }
\email{rhbell@math.upenn.edu}

\author[P. Comparin]{P. Comparin \textsuperscript{2}}
\address{\textsuperscript{2}
Departamento de Matem\'atica y Estad\'istica, Universidad de la Frontera, Av.
Francisco Salazar 1145, Temuco, Chile}
\email{paola.comparin@ufrontera.cl}

\author[J. Li]{J. Li \textsuperscript{3}}
\address{\textsuperscript{3} Department of Mathematics, Princeton University,
Fine Hall, 304 Washington Road, Princeton, NJ}
\email{jenniferli@princeton.edu}

\author[A. Rinc\'{o}n-Hidalgo]{A. Rinc\'{o}n-Hidalgo \textsuperscript{4}} 
\address{\textsuperscript{4} ICTP, Strada Costiera 11, 34151 Trieste, Italy.}
\email{arincon@ictp.it}

\author[A. Sarti]{A. Sarti \textsuperscript{5}} 
\address{\textsuperscript{5} Universit\'{e} de Poitiers, Laboratoire de Math\'{e}matiques et Applications, UMR
7348 du CNRS, 11 bd Marie et
Pierre Curie, 86073 Poitiers Cedex 9,
France%
}
\email{sarti@math.univ-poitiers.fr}

\author[A. Zanardini]{A. Zanardini \textsuperscript{6}}
\address{\textsuperscript{6} Mathematical Institute, Leiden University, Niels Bohrweg 1, Leiden, The Netherlands}
\email{a.zanardini@math.leidenuniv.nl}

\date{\today}

\begin{document}
\maketitle

\begin{abstract}
In this paper we present a complete classification of non-symplectic automorphisms of K3 surfaces whose order is a multiple of seven by describing the topological type of their fixed locus. In the case of purely non-symplectic automorphisms, we provide new results for order 14 and alternative proofs for orders 21, 28 and 42, so that we can unify in the same paper the results on these automorphisms. For each of these orders we also consider not purely non-symplectic automorphisms and obtain a complete characterization of their fixed loci. Several results of our paper were obtained independently in the recent paper \cite{brandhorst2021}
by Brandhorst and Hofmann, but the methods used in the two papers are completely different.
\end{abstract}

\section{Introduction}
\label{sec:introduction}

An automorphism of a K3 surface induces an action on the one-dimensional space of holomorphic 2-forms on the surface, so there are two kinds of automorphisms of K3 surfaces: symplectic and non-symplectic ones.  The automorphism is called \emph{symplectic} if the induced action on the 2-form is trivial. Otherwise, it is called \emph{non-symplectic}, in which case one distinguishes between \emph{purely non-symplectic} automorphisms, meaning the action on the volume form is given by multiplication by a primitive root of unity, and \emph{not purely non-symplectic} automorphisms, meaning some (non-trivial) power of the automorphism is symplectic.   

It is known \cite[Theorem 0.1]{Nikulin} that the rank of the transcendental lattice of a K3 surface carrying a purely non-symplectic automorphism of order $n$ is divisible by the Euler totient function of $n$, which implies $\varphi(n)\leq 20$.  Moreover, all positive integers $n \neq 60$ satisfying such property occur as orders of purely non-symplectic automorphisms by \cite[Main Theorem 3]{MO}. 
For each possible $n$, it is thus a natural and fundamental problem to obtain a complete classification of non-symplectic automorphisms of order $n$ in terms of their fixed locus, and many people have contributed to the development of the subject.  

A classification of non-symplectic automorphisms of prime order $p$ was completed by Nikulin in \cite{Nikulin-inv} when $p=2$, and by  Artebani, Sarti and Taki in \cite{ArtebaniSarti}, \cite{Taki}, \cite{ArtebaniSartiTaki} when $p> 2$. The study of non-symplectic automorphisms of composite order is much more intricate, one of the reasons being that lattice theory works less well in these cases. Results for some possible orders can be found in \cite{ArtebaniSarti4}, \cite{Dillies}, \cite{ACV}, \cite{ACV2}, \cite{Brandhorst}, \cite{AlTabbaaGrossiSarti}, \cite{AlTabbaaSartiTaki}, \cite{AlTabaaSarti-order8-2} and \cite{brandhorst2021}, among others.

In this paper,  we contribute to the classification of non-symplectic automorphisms of orders that are multiples of seven by describing the topological type of their fixed locus. For purely non-symplectic automorphisms, we provide new results for order $14$ and alternative proofs for orders $21, 28$ and $42$, recovering the results in \cite{Brandhorst}. Observe that 42 is the maximum possible order which is a multiple of seven. We also consider the not purely non-symplectic case and obtain a complete characterization for each possible order, which is completely new.

Our main result in the case of purely non-symplectic automorphisms is summarized below in Theorem \ref{main1}. We point the reader to Propositions \ref{thm14}, \ref{thm21}, \ref{thm28} and \ref{thm42} for the details.

\begin{thmINTRO}
Let $\sigma_n$ be a purely non-symplectic automorphism of order $n\in\{14,21,28,42\}$ on a K3 surface $X$. Then the fixed locus of $\sigma_n$ is not empty, and  $\Fix(\sigma_n)$ and the fixed loci of its powers are described by Tables
\ref{tab:fixed}, \ref{tab:21}, \ref{tab:28}, \ref{tab:42}.
\label{main1}
\end{thmINTRO}

We observe that an analogue of Theorem \ref{main1} has also been obtained independently, and via a different method, by Brandhorst and Hofmann in \cite[Theorem 1.4]{brandhorst2021}. The approach we use here is more geometric. In particular, we show the different possibilities for the fixed loci are indeed realizable by explicitly constructing examples that have the desired topological types. Examples in Sections \ref{sec-order21}, \ref{sec-order28} and \ref{sec-order42} were already given in \cite{Brandhorst}, but we provide here a different proof and a more detailed description. 

In the not purely non-symplectic case, we also consider automorphisms of orders 14, 21, 28, and 42 and again we provide a complete classification. In each case we show that not every power of the automorphism can be symplectic and our main result in this direction is  given by Theorem \ref{main2} below. The details are explained in Section \ref{sec-NP}.

\begin{thmINTRO}
Let $\sigma_n$ be a non-symplectic automorphism of order order $n\in\{14,21,28,42\}$ on a K3 surface $X$. 
\begin{enumerate}[(i)]
    \item If $n=14$, then both its square and its 7-th power can be symplectic. In each case, the fixed loci of $\sigma_{14}$ and its powers are described in Propositions \ref{NPorder14} and \ref{prop_NP14-2}.
    \item If $n=21$, its cube is necessarily non-symplectic, whereas $\sigma_{21}^7$ can be symplectic and the fixed loci of $\sigma_{21}$ and its powers in this case are described in Proposition \ref{prop21NP}.
    \item If $n=28$, then $\sigma_n$ is
necessarily purely non-symplectic.
\item If $n=42$, then every power $\sigma_{42}^k$ is necessarily non-symplectic except for $k=14$. In this case, the fixed loci of $\sigma_{42}$ and its powers are described in Proposition \ref{Prop42NP1}.
    \end{enumerate}
\label{main2}
\end{thmINTRO}

To prove Theorems \ref{main1} and \ref{main2} we apply a unified approach to all orders. A central idea consists in observing that the study of the fixed locus of $\sigma_n$ can be reduced to a local analysis of the fixed loci of (some of) its powers. In particular, we rely on the classification result for order $7$ in \cite{ArtebaniSartiTaki}, and some of the tools we use are the Hodge index theorem and the holomorphic and topological Lefschetz formulas (\ref{eq-Lefhol}) and (\ref{eq-Leftop}). Moreover, the examples we construct are often given in terms of elliptic fibrations (see Definition \ref{ellFib}).

The structure of the paper is the following:
Section \ref{sec:background} is devoted to presenting background material, introducing notation and recalling some standard results on automorphisms on K3 surfaces. In Section \ref{sec-order14} we classify purely non-symplectic automorphisms of order 14 in terms of the topological type of their fixed locus. Our main result is outlined in Proposition \ref{thm14} and Tables \ref{tab:fixed} and \ref{tab:14larga}. Moreover, we show the different possibilities indeed occur giving explicit examples. Section \ref{sec-order21} (resp. \ref{sec-order28}, \ref{sec-order42}) provides the classification of purely non-symplectic automophism of order 21 (resp. 28, 42). The topology of their fixed locus is summarized in Tables \ref{tab:21} (resp. \ref{tab:28}, \ref{tab:42}).  In Section \ref{sec-NP} we then consider the case of not purely non-symplectic automoprhisms and obtain a complete characterization for each possible order (14, 21, 28 and 42). Finally, in Section \ref{sec-NS} we study the Néron--Severi lattice of a K3 surface carrying a purely non-symplectic automorphism of order a multiple of seven. 

All computations in this paper are carried out using MAGMA \cite{MAGMA} and we work over $\mathbb{C}$ throughout.

\section*{Acknowledgements}
This work started during the workshop Women in Algebraic Geometry, held virtually at ICERM Providence in July 2020. We thank ICERM for this opportunity. We thank S. Brandhorst for useful comments.
P.C. has been partially supported by Proyecto Fondecyt Iniciaci\'on N.11190428 and Proyecto Fondecyt Regular N.1200608. P.C. and A.S. have been partially supported by Programa de cooperaci\'on cient\'ifica ECOS-ANID C19E06 and Math AmSud-ANID 21 Math 02. A.S. was partially supported by ANR project ANR-20-CE40-0026-0.

\section{Background and notation}\label{sec:background}

A K3 surface is a compact, complex surface which is simply connected and has trivial canonical bundle. An automorphism of finite order on a K3 surface is called non-symplectic if it acts non-trivially on the volume form. The automorphism is called purely non-symplectic if the action is given by multiplication by a primitive $n$-th root of unity.

\begin{notation}Throughout the paper we will adopt the following notations: 
\leavevmode
\begin{itemize}
\item $\omega_X$ will denote a nowhere vanishing holomorphic 2-form on a K3 surface
$X$;
    \item $\zeta_n$ will denote an $n$-th root of unity; 
    \item $\sigma_n$ will denote an automorphism of (finite) order $n$ on a K3 surface $X$.
    In particular, given $\sigma_{n},$ if $m$ divides $n$, we will also denote $\sigma_{n}^{\frac{n}{m}}$ by $\sigma_{m}$; 
    \item $U$ will denote the unique even unimodular hyperbolic lattice of rank 2;
    \item $A_i,D_j,E_6,E_7, E_8, i\geq 1, j\geq 4$ will denote the even, negative definite lattices associated with the
Dynkin diagrams of the corresponding types;
\item $K_7$ will denote the lattice of rank 2 whose bilinear form is given by the matrix $\left(\begin{array}{cc}
    -4 & 1 \\
     1& -2
\end{array}\right)$;
\item given a lattice $L$, 
$L(k)$ will denote the lattice having as bilinear form the one on
$L$ multiplied by $k$, $k\in\mathbb Z$;
\item $S(\sigma_n)$ will denote the invariant lattice: $\{x\in H^2(X,\mathbb Z) |\ (\sigma_n)^{*}(x)=x\}$, which is primitively embedded in the N\'eron--Severi lattice $\rm{NS}(X)$ of the surface X, by \cite{Nikulin}.

\end{itemize}

\end{notation}

Given any purely non-symplectic automorphism $\sigma_n$ with $n\geq 3$, by the Hodge Index Theorem, its fixed locus $\Fix(\sigma_n)$ consists of a disjoint union of smooth curves and isolated points:
\begin{equation}
\label{fixed}\Fix(\sigma_n)=C_{g_n}\sqcup R_1\sqcup\ldots \sqcup R_{k_n}\sqcup\{p_1,\ldots,p_{N_n}\}\end{equation}
where $C_{g_n}$ is a smooth curve of genus $g_n\geq 0$ and $R_i$ are rational curves and $p_i$ are isolated fixed points, whose total number is $N_{n}$.

By \cite{Nikulin}, the action of $\sigma_n$ can be locally linearized and diagonalized around a fixed point so that $\sigma_n$ acts as multiplication by the matrix 
\[
A_{i,n} :=
\begin{bmatrix}
\zeta_{n}^{1+i}     & 0 \\
0     & \zeta_{n}^{n-i}
\end{bmatrix} 
\text{ such that }0 \leq i < n,
\]
and we say that such a fixed point is of type $A_{i,n}$. The total number of fixed points of type $A_{i,n}$ will be denoted by $m_{i,n}$.
Observe that if $i=0$, one of the eigenvalues of $A_{0,n}$ is 1, thus the fixed point is not isolated but it belongs to a fixed curve.

We may use the holomorphic Lefschetz formula for $\sigma_n$ to compute the Lefschetz number $L(\sigma_n)$ in two ways. First of all, we have:
\[L(\sigma_n)=\sum_{i=0}^2(-1)^i{\rm tr} (\sigma_n^*|_{H^i(X,\mathcal O_X)})=1+\zeta_n^{n-1}\]
where we are assuming $\sigma_n^*\omega_X=\zeta_n \omega_X$.
On the other hand, we have:
\[L(\sigma_n)=\sum_{i=1}^{n-2}\frac{m_{i,n}}{\det(I-\sigma_n^*|T_X)}+\alpha_n\frac{1+\zeta_n}{(1-\zeta_n)^2}.\]
where $\displaystyle{\alpha_n\coloneqq\sum_{C\subset\Fix(\sigma_n)}(g(C)-1)}$. Equating these two expressions we obtain a linear system of equations that allows us to determine the possible values for $m_{i,n}$ and $\alpha_n$:
\begin{equation}\label{eq-Lefhol}
    1+\zeta_n^{n-1}=\sum_{i=1}^{n-2}\frac{m_{i,n}}{(1-\zeta_n^{1+i})(1-\zeta_n^{n-i})}+\alpha_n\frac{1+\zeta_n}{(1-\zeta_n)^2}.
\end{equation}

The topological Lefschetz formula, in turn, can be used to compute the Euler characteristic of the fixed locus of $\sigma_n$:
\begin{equation}\label{eq-Leftop}
    \chi_n \doteq \chi(\Fix(\sigma_n))=2+{\rm tr}(\sigma_n^*|H^2(X,\mathbb R)).
\end{equation}

Both (\ref{eq-Lefhol}) and (\ref{eq-Leftop}) will be used extensively  throughout the paper in order to perform a local analysis of the action of non-symplectic automorphisms with order a multiple of seven. It is this local analysis that will lead us to a complete classification of such automorphisms, in terms of the topological type of their fixed locus.

We will also make extensive use of the already known classification of non-symplectic automorphisms of order seven, given by Theorem \ref{order7} below:  

\begin{theorem} \cite[Section 6]{ArtebaniSartiTaki}
If $X$ is a K3 surface and $\sigma_7$ a non-symplectic automorphism of order 7, then the possibilities for the fixed locus of $\sigma_7$ and the invariant lattice $S(\sigma_7)$ are listed in Table \ref{tab:7} and all cases exist. 
\begin{table}[H]
\centering
\begin{tabular}{c!{\vrule width 1.5pt}c|c|c|c|c|c}
&$m_{1,7}$&$m_{2,7}$&$m_{3,7}$&$g_7$&$k_7$&$S(\sigma_7)$\\
\noalign{\hrule height 1.5pt}
A&2&1&0&1&0&$U\oplus K_7$\\ \hline
$\dagger$&2&1&0&-&-&$U(7)\oplus K_7$\\
\hline
B&4&3&1&1&1&$U\oplus E_8$\\ \hline
C&4&3&1&0&0&$U(7)\oplus E_8$\\
\hline
D&6&5&2&0&1&$U\oplus E_8\oplus A_6$
\end{tabular}
\caption{Order 7}
\label{tab:7}
\end{table}
\label{order7}
\end{theorem}

For each possibility in our classification, the existence of a K3 surface carrying an automorphism with fixed locus having the desired topological type will then be obtained via the construction of explicit examples. Most of the examples will arise from elliptic fibrations. Therefore, we also recall some generalities about elliptic K3 surfaces, and we refer the reader to \cite{Miranda} for details.

\begin{defin}
An elliptic fibration on a projective surface $X$ consists of a surjective proper morphism $\pi:X \to C$ (with connected fibers) such that the  generic fiber is a smooth curve of genus one, and we further assume there exists a section $s:C \to X$ (i.e. $\pi \circ s= id_{C}$). 
\label{ellFib}
\end{defin}

A K3 surface $X$ admits an elliptic fibration if and only if there exists a primitive embedding of the hyperbolic lattice $U$ into $NS(X)$, the Néron--Severi lattice of the surface. Any elliptic fibration can be reconstructed from its Weierstrass model, and in the case of K3 surfaces such model is given by an equation of the form:
\begin{equation}
y^2=x^3+A(t)x+B(t), \quad t\in \mathbb{P}^1
\label{weierstrass}
\end{equation}
where $A(t)$ and $B(t)$ are polynomials in $\CC[t]$ of degrees $8$ and $12$, respectively. 

Given an elliptic fibration, a chosen section $s:C \to X$ is called the zero section; and one identifies the map $s$ with the curve $s(C)$ on $X$. In the model given by (\ref{weierstrass}), the zero section is $t \mapsto (0:1:0)$.

We further observe that, using (\ref{weierstrass}), the volume form can  be   written locally as 
\[
\frac{dx \wedge dt}{2y}
\]
Moreover, the discriminant of the fibration is the polynomial of degree $24$:
\[
\Delta(t)=4A(t)^3+27B(t)^2
\]
and each zero of $\Delta(t)$ corresponds to a singular fiber of the fibration. The possible singular fibers have been classified by N\'eron and Kodaira \cite{neron}, \cite{kodaira1}, \cite{kodaira2}.

\section{Order 14}\label{sec-order14}
Let $\sigma_{14}$ be a purely non-symplectic automorphism of order 14. As described in Section \ref{sec:background}, the local actions of $\sigma_{14}$ at fixed points are of seven types. Points of type $A_{0,14}$ lie on a fixed curve, and isolated fixed points are of type $A_{i,14}$ for $i=1,\ldots,6$. Thus, the fixed locus of $\sigma_{14}$ can contain both fixed curves and isolated fixed points of six different types. The goal of this section is to prove the following classification result:
\begin{proposition} The fixed locus of a purely non-symplectic automorphism of order $14$ on a K3 surfaces is not empty and it consists of either:
\begin{enumerate}[(i)]
    \item The union of $N_{14}$ isolated points, where $N_{14}\in \{3,4,5,6,7\}$; or
    \item The disjoint union of a rational curve and $N_{14}$ isolated points, where $N_{14}\in\{6,11,12\}$.
\end{enumerate}
Moreover, all these possibilities occur, and in each case $\sigma_7\doteq \sigma_{14}^2$ fixes at least one curve. A more detailed description is given in Tables \ref{tab:fixed} and \ref{tab:ex} below, where $\sigma_2$ denotes the involution $\sigma_{14}^7$.
\label{thm14}

{\small
\begin{table}[H]\centering
\begin{tabular}{c!{\vrule width 1.5pt}c|c|c|c}
&$\Fix(\sigma_{14})$&$\Fix(\sigma_7)$&$\Fix(\sigma_2)$&Example\\
\noalign{\hrule height 1.5pt}
A1(9,1)&$\{p_1,\ldots,p_7\}$&$E\sqcup\{p_1,p_2,p_3\}$&$C_9\sqcup R$&\ref{example_A1_(9,1)}\\

A1(3,2)&$\{p_1,\ldots,p_7\}$&$E\sqcup\{p_1,p_2,p_3\}$&$C_3\sqcup R_1\sqcup R_2$& \ref{example_A1_(3,2)}\\

\hline
A2&$\{p_1,\ldots,p_5\}$&$E\sqcup\{p_1,q_1,q_2\}$&$C_9$& \ref{example_A2_(9,0)}\\

\noalign{\hrule height 1.5pt}
B3&$R\sqcup\{p_1,\ldots,p_{12}\}$&$E\sqcup R\sqcup\{p_1,\ldots p_8\}$&$C_6\sqcup R\sqcup R_1\sqcup\ldots\sqcup R_4$ &\ref{example_B3}\\
\noalign{\hrule height 1.5pt}

C1(6,1)&$\{p_1,\ldots,p_6\}$&$R\sqcup\{p_1,\ldots p_4,q_1,\ldots,q_4\}$&$C_6\sqcup R'$&\ref{example_C1_(6,1)}\\

C1(7,2)&$\{p_1,\ldots,p_6\}$&$R\sqcup\{p_1,\ldots p_4,q_1,\ldots,q_4\}$&$C_7\sqcup R_1\sqcup R_2$&\ref{example_C1_(7,2)}\\

C1(0,2)&$\{p_1,\ldots,p_6\}$&$R\sqcup\{p_1,\ldots p_4,q_1,\ldots,q_4\}$&$R_1\sqcup R_2 \sqcup R_3$&\ref{example_C1_(0,2)}\\

\hline

C2&$\{p_1,\ldots,p_4\}$&$R\sqcup\{p_1,p_2,q_1\ldots q_6\}$&$C_6$&\ref{example_C2(6,0)}\\
\hline
C3&$R\sqcup\{p_1,\ldots,p_6\}$&$R\sqcup\{p_1,\ldots p_6,q_1,q_2\}$&$C_6\sqcup R\sqcup R'$&\ref{example_C3_(6,2)}\\
\noalign{\hrule height 1.5pt}
D2&$\{p_1,p_{2},p_{3}\}$&$R_1\sqcup R_2\sqcup\{p_1,\ldots,p_{13}\}$&$C_3$&\ref{example_D2}\\\hline
D3&$\{p_1,\ldots,p_{7}\}$&$R_1\sqcup R_2\sqcup\{p_1,p_2,p_3,q_1\ldots,q_{10}\}$&$C_3\sqcup R'\sqcup R''$&\ref{example_D3}\\\hline


D8&$R\sqcup\{p_1,\ldots,p_{11}\}$&
$R\sqcup R'\sqcup\{p_1,\ldots,p_{9},q_1,\ldots,q_4\}$&
$C_3\sqcup R\sqcup R_1\sqcup\ldots R_4$&\ref{example_D8}\\
\end{tabular}
\caption{Order 14}
\label{tab:fixed}
\end{table}
}

\end{proposition}

The proof of Proposition \ref{thm14} is done in several steps. First, in Section \ref{sec-poss} we use formulas \eqref{eq-Lefhol} and \eqref{eq-Leftop} in order to generate Table \ref{tab:14larga}, which provides a list of possibilities for the fixed locus of $\sigma_{14}$ and its powers. In Section \ref{sec-exclude} we then exclude many of these possibilities using geometric arguments, and produce a new table - Table \ref{tab:ex}. Finally, in Section \ref{sec-ex} we show all the remaining cases listed in Table \ref{tab:ex} are indeed admissible by constructing explicit examples that have the desired topological types.

\subsection{Generation of table of possibilities}\label{sec-poss}

Since $\sigma_{14}$ is purely non-symplectic, its square $\sigma_7 :=\sigma_{14}^2$ is a non-symplectic automorphism of order 7. Moreover, $\Fix(\sigma_{14}) \subseteq \Fix(\sigma_7)$ and in particular each curve contained in  $\Fix(\sigma_{14})$ is also contained in $\Fix(\sigma_7)$.  

Now, for all $i=1,\ldots,6$ we have that $(A_{i, 14})^{2} = A_{j,7}$ for some $j\in\{0,1,2,3\}$. 
For instance, $A_{1, 14}^{2} = A_{1, 7}$.
Thus, fixed points of $\sigma_{14}$ that are of type $A_{1,14}$ are also points of type $A_{1,7}$ for $\sigma_7$.
Similarly:
\begin{itemize}
\item points of type $A_{5,14}$ for $\sigma_{14}$ are of type $A_{1,7}$ for $\sigma_7$,
\item points of types $A_{2,14}$ and $A_{4,14}$ for $\sigma_{14}$ are of type $A_{2,7}$ for $\sigma_7$, and 
\item points of type $A_{3,14}$ for $\sigma_{14}$ are of type $A_{3,7}$ for $\sigma_7$.
\end{itemize}


In particular, the following inequalities hold:

\begin{equation}\label{rel_7_14}
\begin{cases} m_{1,14} + m_{5,14}  & \leq   m_{1,7} \\
m_{2,14} + m_{4,14} & \leq   m_{2,7} \\
m_{3,14}& \leq  m_{3,7}
\end{cases}
\end{equation}

And we further observe the following:
\begin{remark}\label{m_6}
Note that $A_{6, 14}^{2} = A_{0, 7}$, 
which shows that points of type $A_{6,14}$ lie on a curve fixed by $\sigma_{7}$. Therefore, if $m_{6,14}\neq 0$, then there are curves in $\Fix(\sigma_7)$ which are not in $\Fix(\sigma_{14})$.
\end{remark}  

\begin{remark}\label{types}
A rational curve $R$ invariant for an automorphism $\sigma_n$ is either pointwise fixed or $R$ admits two isolated fixed points. In the latter case, the points are of consecutive types, i.e., if one point is of type $A_{i,n}$, then the other is of type $A_{i+1,n}$. If $n=14$, as in \cite[Lemma 4]{ArtebaniSarti4}, one can prove that, given a tree of rational curves invariant for $\sigma_{14}$, the distribution of types of isolated fixed points is as shown in Figure \ref{fig:tree}. 
This can be done in a similar way for $n=21,28,42$.
\end{remark}


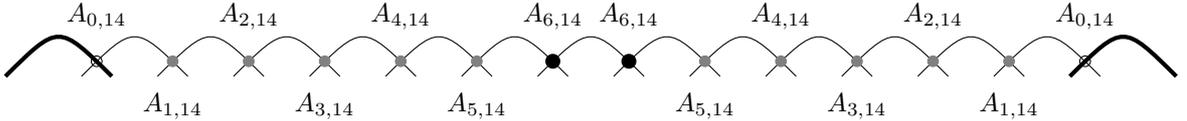
\begin{figure}[H]
\centering
\begin{tikzpicture}
\draw[ultra thick] (-0.2,-0.2) .. controls (0.5,0.5) .. (1.2,-0.2);
\draw (0.8,-0.2) .. controls (1.5,0.5) .. (2.2,-0.2);
\draw (1.8,-0.2) .. controls (2.5,0.5) .. (3.2,-0.2);
\draw (2.8,-0.2) .. controls (3.5,0.5) .. (4.2,-0.2);
\draw (3.8,-0.2) .. controls (4.5,0.5) .. (5.2,-0.2);
\draw (4.8,-0.2) .. controls (5.5,0.5) .. (6.2,-0.2);
\draw (5.8,-0.2) .. controls (6.5,0.5) .. (7.2,-0.2);
\draw (6.8,-0.2) .. controls (7.5,0.5) .. (8.2,-0.2);
\draw (7.8,-0.2) .. controls (8.5,0.5) .. (9.2,-0.2);
\draw (8.8,-0.2) .. controls (9.5,0.5) .. (10.2,-0.2);
\draw (9.8,-0.2) .. controls (10.5,0.5) .. (11.2,-0.2);
\draw (10.8,-0.2) .. controls (11.5,0.5) .. (12.2,-0.2);
\draw (11.8,-0.2) .. controls (12.5,0.5) .. (13.2,-0.2);
\draw (12.8,-0.2) .. controls (13.5,0.5) .. (14.2,-0.2);
\draw[ultra thick] (13.8,-0.2) .. controls (14.5,0.5) .. (15.2,-0.2);
\draw [black] (1,0) circle (2pt);
\filldraw [gray] (2,0) circle (2pt);
\filldraw [gray] (3,0) circle (2pt);
\filldraw [gray] (4,0) circle (2pt);
\filldraw [gray] (5,0) circle (2pt);
\draw [black] (14,0) circle (2pt);
\filldraw [black, ultra thick] (7,0) circle (2pt);
\filldraw [black, ultra thick] (8,0) circle (2pt);
\filldraw [gray] (9,0) circle (2pt);
\filldraw [gray] (10,0) circle (2pt);
\filldraw [gray] (11,0) circle (2pt);
\filldraw [gray] (12,0) circle (2pt);
\filldraw [gray] (13,0) circle (2pt);
\filldraw [gray] (6,0) circle (2pt);
\draw[black] (1,0.6) node {$A_{0,14}$};
\draw[black] (2,-0.6) node {$A_{1,14}$};
\draw[black] (3,0.6) node {$A_{2,14}$};
\draw[black] (4,-0.6) node {$A_{3,14}$};
\draw[black] (5,0.6) node {$A_{4,14}$};
\draw[black] (6,-0.6) node {$A_{5,14}$};
\draw[black] (7,0.6) node {$A_{6,14}$};
\draw[black] (8,0.6) node {$A_{6,14}$};
\draw[black] (9,-0.6) node {$A_{5,14}$};
\draw[black] (10,0.6) node {$A_{4,14}$};
\draw[black] (11,-0.6) node {$A_{3,14}$};
\draw[black] (12,0.6) node {$A_{2,14}$};
\draw[black] (13,-0.6) node {$A_{1,14}$};
\draw[black] (14,0.6) node {$A_{0,14}$};
\end{tikzpicture}
\caption{Actions of $\sigma_{14}$ and $ \sigma_7$ on a tree of rational curves. Thin curves are invariant but not pointwise fixed. Thick curves are pointwise fixed by $\sigma_{14}$. The gray points are isolated fixed points for both $\sigma_{14}$ and $\sigma_7$, and the two black points in the middle lie on a curve fixed by $\sigma_7$ only.}
\label{fig:tree}
\end{figure}

As a consequence, from \eqref{rel_7_14} and the previous remarks, if we apply formula (\ref{eq-Lefhol}) to $\sigma_{14}$ we obtain the following linear system of equations:
\begin{equation}
    \begin{cases}
    m_{1,14}&=4\alpha_{14}-2m_{4,14}+m_{5,14}\\
    m_{2,14}&=1-2m_{5,14}+3m_{4,14}\\
    m_{6,14}&=8m_{4,14}+4-2m_{3,14}-2\alpha_{14}-4m_{5,14}
    \end{cases}
    \label{Lef-hol}
\end{equation}


This allows us to prove the following two Lemmas.
\begin{lemma}\label{alpha_no2}
The value of $\alpha_{14}$ is either 0 or 1.
\end{lemma}
\begin{proof}
Since $\Fix(\sigma_{14})\subset\Fix(\sigma_7)$, a curve that is pointwise fixed by $\sigma_{14}$ must be contained in $\Fix(\sigma_7)$. 
Thus, according to Table \ref{tab:7}, we must have $\alpha_{14}\in\{0,1,2\}$.
Assume $\alpha_{14} = 2$. Then $\sigma_{14}$ fixes at least two rational curves. Therefore the fixed locus under $\sigma_7$ is described by the last row of Table \ref{tab:7}, and both rational curves in $\Fix(\sigma_7)$ are fixed by $\sigma_{14}$. By Remark \ref{m_6}, $m_{6,14}=0$. But  plugging in $\alpha_{14}=2$ and $m_{6,14}=0$ with the inequalities \eqref{rel_7_14} with the values of $m_{1,7}, m_{2,7}, m_{3,7}$ from last line of Table \ref{tab:7} into \eqref{Lef-hol} 
yields an unsolvable system. Therefore $\alpha_{14}$ can only be equal 0 or 1.
\end{proof}

\begin{lemma}
There is no purely non-symplectic automorphism $\sigma_{14}$ of order 14 such that the fixed locus of $\sigma_7$ is described by the second row of Table \ref{tab:7}.
\end{lemma}
\begin{proof}
In this case, no curves are fixed by $\sigma_7$ and hence no curves are fixed by $\sigma_{14}$. Thus $\alpha_{14}=0$. 
But a MAGMA calculation shows that in this case a solution of \eqref{Lef-hol} would have $m_{6,14} = 4$, which would imply that $\sigma_{14}$ fixes a curve, and so this case cannot occur. 
\end{proof}

In fact we can completely describe  what are the possible solutions to (\ref{Lef-hol}), i.e. what are the possibilities for the vector $\underline m=(m_{1,14},m_{2,14}, m_{3,14}, m_{4,14}, m_{5,14}, m_{6,14})$ and for the value of $\alpha_{14}$. Organizing the possibilities according to the fixed locus of
$\sigma_7=\sigma_{14}^2$, we prove:

\begin{proposition}
If $\sigma_{14}$ is a purely non-symplectic automorphism of order $14$ on a K3 surface, then the possible vectors $\underline m=(m_{1,14},m_{2,14}, m_{3,14}, m_{4,14}, m_{5,14}, m_{6,14})$ satisfying \eqref{Lef-hol} are listed in Table \ref{tab:14} below. The symbol $*$ means that the action on the elliptic curve $E$ is a translation.
\label{vectors}
\end{proposition}

In particular, we obtain a list of possibilities for the fixed locus of $\sigma_{14}$.

{\small
\begin{table}[H]\centering
\begin{tabular}{c|cccccc|c|c}
&$m_{1,14}$&$m_{2,14}$&$m_{3,14}$&$m_{4,14}$&$m_{5,14}$&$m_{6,14}$&$\alpha_{14}$& curves fixed by $\sigma_{14}$\\
\hline
A1&0&0&0&1&2&4&0&$\emptyset$\\
A2&0&1&0&0&0&4&0&$\emptyset$\\
\hline
B1&0&0&1&1&2&2&0&$E$ \\
B1*&0&0&1&1&2&2&0& $\emptyset$\\
B2&0&1&1&0&0&2&0&$E$ \\
B2*&0&1&1&0&0&2&0& $\emptyset$\\

B3&3&2&1&1&1&4&1&$R$\\
B4&4&1&1&0&0&0&1&$R\sqcup E$ \\
B4*&4&1&1&0&0&0&1&$R$ \\

\hline
C1&0&0&1&1&2&2&0&$\emptyset$\\
C2&0&1&1&0&0&2&0&$\emptyset$\\
C3&4&1&1&0&0&0&1&$R$\\
\hline
D1&0&0&2&1&2&0&0&$\emptyset$\\
D2&0&1&2&0&0&0&0&$\emptyset$\\
D3&0&0&0&1&2&4&0&$\emptyset$\\
D4&0&1&0&0&0&4&0&$\emptyset$\\
D5&4&0&0&1&2&2&1&$R$\\
D6&4&1&0&0&0&2&1&$R$\\
D7&3&1&2&2&3&2&1&$R$\\
D8&3&2&2&1&1&2&1&$R$\\
\end{tabular}
\caption{}
\label{tab:14}
\end{table}}


\begin{proof}[Proof of Proposition \ref{vectors}]
We consider each row of Table \ref{tab:7}:
\begin{enumerate}
    \item[{\bf Case A}] This corresponds to the case in which the fixed locus of $\sigma_7$ is described by the first row of Table \ref{tab:7} and $\Fix(\sigma_7)$ consists of a genus one curve $E$, so we only need to determine whether $\sigma_{14}$ itself fixes $E$. In both cases, $\alpha_{14}=0$ and by \eqref{rel_7_14} $m_{3,14}=0$. A MAGMA calculation shows that the only vectors $\underline m$
which satisfy \eqref{Lef-hol} with $\alpha_{14}=m_{3,14}=0$ are $(0, 0, 0, 1, 2, 4)$ and $(0, 1, 0, 0, 0, 4)$.
By Remark \ref{m_6}, $\sigma_{14}$ does not fix $E$.

\item[{\bf Case B}] When $\Fix(\sigma_7)$ is described by the third row of Table \ref{tab:7}, the automorphism $\sigma_7$ fixes a genus one curve $E$ and a rational curve $R$.
We analyze this case by considering the possibilities for $\alpha_{14}$ and $m_{6,14}$.

First, suppose $\Fix(\sigma_{14})$ contains no curves, so $\sigma_{14}$ fixes neither $R$ nor $E$; in this case, $\alpha_{14} = 0$. Since $\sigma_{14}$ acts as an involution on $E$, by the Riemann-Hurwitz formula it has either four fixed points (coming from $P \mapsto -P$ after a choice of point at infinity) or no fixed points (coming from $P \mapsto P + T$ where $T$ is a 2-torsion point). The action on $R$ has 2 fixed points, so $m_{6,14}$ is either 6 or 2. A MAGMA calculation applying the constraints from \eqref{Lef-hol} shows that the possibilities for $\underline m$ are 
$(0, 0, 1, 1, 2, 2)$ and $(0, 1, 1, 0, 0, 2)$.

Second, suppose that $E\subset\Fix(\sigma_{14})$ and $R\not\subset\Fix(\sigma_{14})$; in this case, $\alpha_{14} = 0$ and $m_{6,14} = 2$. The possibilities for $\underline m$ in this case are $(0, 0, 1, 1,2,2)$ and $(0,1,1,0,0,2)$.

Next, if $R\subset\Fix(\sigma_{14})$ and $E\not\subset\Fix(\sigma_{14})$, $\sigma_{14}$ fixes either none or four points on $E$, so $\alpha_{14} = 1$ and $m_{6,14} = 0$ or $4$, and the possibilities for $\underline m$ are $(3, 2, 1,1,1,4)$ and $(4,1,1,0,0,0)$.

Lastly, if $ E \sqcup R\subset\Fix(\sigma_{14})$, all curves fixed under $\sigma_7$ are also fixed under $\sigma_{14}$, so $\alpha_{14} = 1$ and $m_{6,14}=0$, and the only possibility is $\underline m=(4,1,1,0,0,0)$.

\item[{\bf Case C}] In this case, the only curve fixed by $\sigma_7$ is a rational curve $R$. If $\sigma_{14}$ does not fix $R$, then $\alpha_{14} = 0$ and $m_{6,14} = 2$ and the solutions of \eqref{Lef-hol} for $\underline m$ are
$(0, 0, 1, 1, 2, 2)$ and $(0, 1, 1, 0, 0, 2)$.

On the other hand, if $\sigma_{14}$ fixes $R$, then $\alpha_{14} = 1$ and $m_{6,14} = 0$, and the only possibility for $\underline m$ is $(4, 1, 1, 0, 0, 0)$.

\item[{\bf Case D}] Finally, if the fixed locus of $\sigma_7$ is described by the last row of Table \ref{tab:7}, 
the curves fixed by $\sigma_7$ are two rational curves $R_1 \sqcup R_2$. 
First, suppose neither $R_1$ nor $R_2$ is fixed by $\sigma_{14}$, thus $\alpha_{14} = 0$. Then, either $\sigma_{14}$ exchanges $R_1$ and $R_2$, or $\sigma_{14}$ acts nontrivially on $R_1$ and $R_2$. If $R_{1}$ and $R_{2}$ are exchanged (hence fixing no points on either curve), then $m_{6,14}=0$ and the possibilities for $\underline m$ are $(0, 0, 2, 1, 2, 0)$ and $(0, 1, 2, 0, 0, 0)$.
Otherwise, there are a total of 4 points fixed on these curves, so $m_{6,14} = 4$ and the possibilities for $\underline m$ are
$(0, 0, 0, 1, 2, 4)$ and $(0, 1, 0, 0, 0, 4)$.

If $\sigma_{14}$ fixes one rational curve and acts nontrivially on the other, $\alpha_{14} = 1$ and $m_{6,14} = 2$. Possibilities for $\underline m$ are $(0, 0, 1, 1, 2, 2)$, $(0, 1, 1, 0, 0, 2)$, $(3,2,2,1,1,2)$ and $(3,1,2,2,3,2)$.
By Lemma \ref{alpha_no2}, $\sigma_{14}$ does not fix both $R_1$ and $R_2$.
\end{enumerate}
\end{proof}

We also observe the following:

\begin{proposition}\label{prop:B}
If $\sigma_{14}$ is a purely non-symplectic automorphism on a $K3$ surface such that $\sigma_7=\sigma_{14}^2$ is of type $B$ (see Table \ref{tab:7}), then $\sigma_{14}$ is of type $B3$.
\end{proposition}

\begin{proof}
Let $X$ be a K3 surface and $\sigma_{14}$ a purely non--symplectic automorphism of order 14 acting on $X$. Assume we are in case B so that $\sigma_7$ fixes a genus 1 curve, a rational curve and eight isolated points.
By \cite[Thm. 6.3]{ArtebaniSartiTaki} $X$ admits an elliptic fibration with a reducible fiber of types $II^*$ at $t=\infty$, a smooth fiber at $t=0$ and $14$ singular fibers of type $I_1$. The automorphism $\sigma_7$ fixes the fiber over 0 and the central component of the fiber $II^*$; all eight isolated points of $\sigma_7$ lie on the fiber $II^*$. 

Since $\sigma_7$ fixes the genus one curve, the fibration is $\sigma_{7}$-invariant. Thus the fibers over $t=0$ and $t=\infty$ are preserved. The $II^*$ fiber does not admit a reflection, and so we can conclude that the central component must be fixed by $\sigma_{14}$. Moreover, the eight isolated fixed points of $\sigma_7$ are also isolated and fixed by $\sigma_{14}$. Table \ref{tab:14} shows that the only case with $N_{14}\geq8$ is case B3. We also observe that, because $m_{6,14}=4$, the automorphism $\sigma_{14}$ acts as an involution on the genus one curve with four fixed points. 
\end{proof}

Now, in order to better understand the different fixed loci listed in Table \ref{tab:14}, the next step in our approach consists in further studying the fixed locus of the involution $\sigma_{14}^7$, and the eigenspaces of $\sigma_{14}^*$ in $H^2(X,\CC)$. We use the following notation:
\[ d_i\coloneqq\dim H^2(X,\CC)_{\zeta_{i}}, i=1,2,7,14.\]
In particular, we have
\[22=6d_{14}+6d_7+d_2+d_1.\]
\begin{remark}\label{d14}
Observe that $\rk S(\sigma_{14})=d_1$ and $\rk S(\sigma_7)=d_2+d_1$ and $\rk S(\sigma_2)=6d_7+d_1$.
\end{remark}

Moreover, by applying the topological Lefschetz formula \eqref{eq-Leftop} to the fixed loci of $\sigma_{14}$ and its powers, we obtain the following system of equations:
\begin{equation}
\begin{cases}
\chi_{14}\coloneqq\chi(\Fix(\sigma_{14}))= 2+d_{14}-d_7-d_2+d_1\\
\chi_{7}\coloneqq\chi(\Fix(\sigma_7))= 2-d_{14}-d_7+d_2+d_1\\
\chi_{2}\coloneqq\chi(\Fix(\sigma_2))= 2-6d_{14}+6d_7-d_2+d_1\\
\end{cases}
\label{Lef-top}
\end{equation}

Using \eqref{Lef-top} and Table \ref{tab:14} we can thus obtain a list of  possibilities for $(d_{14},d_7,d_2,d_1)$ as well as the corresponding Euler characteristics $(\chi_{14}, \chi_7,\chi_2)$. We present our results in Table \ref{tab:14larga} below. 


{\small
\begin{table}[H]
\begin{tabular}{c!{\vrule width 1.5pt}c|c|c|c|c|c|c}
&$N_{14}$ & $\alpha_{14}$ & $\chi_{14}$&$\chi_7$&$\chi_2$& $(d_{14},d_7,d_2,d_1)$&Possible $(g_2,k_2)$\\
\noalign{\hrule height 1.5pt}
A1&7&0&7&3&-14& (3,0,1,3) &$(8,0),(9,1)$\\
&&&&&0& (2,1,0,4) &$(1,0),(2,1),(3,2), (4,3),(5,4),(6,5)$\\
\hline

A2&5&0&5&3&-16& (3,0,2,2) &$(9,0), (10,1)$\\
&&&&&-2& (2,1,1,3) &$(2,0),(3,1),(4,2),(5,3),(6,4)$\\
&&&&&12& (1,2,0,4) &$(0,5), (1,6), (2,7)$\\
\noalign{\hrule height 1.5pt}

B3&12&1&14&10&0& (2,0,0,10)&$(1,0),(2,1),(3,2),(4,3),(5,4),(6,5)$\\
\noalign{\hrule height 1.5pt}
C1&6&0&6&10&-8& (2,0,4,6) &$(5,0), (6,1),(7,2)$\\
&&&&&6& (1,1,3,7) &$(0,2), (1,3),(2,3),(3,5)$\\\hline
C2&4&0&4&10&-10& (2,0,5,5) &$(6,0),(7,1)$\\
&&&&&4& (1,1,4,6) & $(0,1),(1,2),(2,3),(3,4),(4,5)$\\\hline

C3&6&1&8&10&-6& (2,0,3,7) &$(4,0),(5,1),(6,2)$\\
&&&&&8& (1,1,2,8) &$(0,3),(1,4),(2,5),(3,6)$\\
\noalign{\hrule height 1.5pt}
D1&5&0&5&17&-2& (1,0,7,9) &$(2,0),(3,1),(4,2),(5,3),(6,4)$\\\hline
D2&3&0&3&17&-4& (1,0,8,8)&$(3,0),(4,1),(5,2),(6,3)$\\\hline
D3&7&0&7&17&0& (1,0,6,10)&$(1,0),(2,1),(3,2),(4,3),(5,4),(6,5)$\\\hline
D4&5&0&5&17&-2&(1,0,7,9)&$(2,0),(3,1),(4,2),(5,3),(6,4)$\\\hline
D5&9&1&11&17&4&(1,0,4,12)&$(0,1),(1,2),(2,3),(3,4),(4,5)$\\\hline
D6&7&1&9&17&2&(1,0,5,11)&$(0,0),(1,1), (2,2),(3,3),(4,4),(5,5)$\\\hline
D7&13&1&15&17&8&(1,0,2,14)&$(0,3),(1,4),(2,5),(3,6)$\\\hline
D8&11&1&13&17&6&(1,0,3,13)&$(0,2), (1,3),(2,3),(3,5)$\\
\end{tabular}
\caption{}
\label{tab:14larga}
\end{table}
}


We remark that by \cite{Nikulin-inv}, the fixed locus of a non-symplectic involution is either empty; or it consists of two disjoint elliptic curves; or
\begin{equation}
\Fix(\sigma_2)=C_{g_2}\sqcup R_1\sqcup\ldots \sqcup R_{k_2}
\label{fix-inv}
\end{equation}

where $C_{g_2}$ is a smooth curve of genus $g_2\geq 0$ and $R_i$ are rational curves, and all possibilities for the pair of invariants $(g_2,k_2)$ are classified (see for example \cite[Figure 1]{ArtebaniSartiTaki}).

In our case, it follows from (\ref{Lef-hol}) that $\Fix(\sigma_2)$ cannot be empty. Any possible solution to (\ref{Lef-hol}) gives us that $\Fix(\sigma_{14})$ contains at least one fixed point. In fact, this also implies $\Fix(\sigma_2)$ cannot be the union of two elliptic curves either.  If the latter occurs, then  the action of $\sigma_{14}$ on each elliptic curve would be without fixed points. Since $\Fix(\sigma_{14}) \subset \Fix(\sigma_{2})$, again we would have no fixed points in $\Fix(\sigma_{14})$, contradicting (\ref{Lef-hol}).
As a consequence, for each line of Table \ref{tab:14larga}, we know that $\Fix(\sigma_{2})$ is of the form (\ref{fix-inv}). Moreover, there is more than one possible pair of invariants $(g_2,k_2)$. 

\subsection{Excluding cases}\label{sec-exclude}
We will now show many cases of Table \ref{tab:14larga} can actually be excluded for geometric reasons. We prove a series of Lemmas in this direction.

\begin{notation}
In what follows, we will use the following notation: A1(8,0) means that $\Fix(\sigma_{7})$ is as in line A of Table \ref{tab:7}, $\Fix(\sigma_{14})$ is described in the line A1 of Table \ref{tab:14larga} and $(g_2,k_2)=(8,0)$. Similarly for all other cases.
\end{notation}
\begin{remark}
Observe that cases B3(1,0) and D6(0,0) are not admissible because in both cases, the fixed locus $\Fix(\sigma_{14})$ contains a rational curve while $\Fix(\sigma_2)$ does not.
\end{remark}
\begin{remark}
$\Fix(\sigma_2)$ does not contain a curve of genus 2, 4 or 5. This is a direct consequence of the following Lemma.
\end{remark}

\begin{lemma}[\cite{homma}]\label{homma}
Let $C$ be a curve of genus $g\geq 2$ that admits an automorphism of prime order $q$ where $q>g$. Then either $q=g+1$ or $q=2g+1$.
\end{lemma}

\begin{lemma}
The following cases are not admissible:
 \[A1(6,5),\ A1 (8,0),\ A1(1,0),\ A1 (6,5), \ A2(6,4), \ A2 (0,5),\ A2 (1,6),\ B3 (3,2),\]
 \[ C1 (3,5),\ C2(3,4),\ C3(3,6),\ D1 (6,4),\ D2 (6,3),\ D3 (1,0),\ D3 (6,5),\]
 \[ D4 (6,4),\ D5 (0,1),\ D5(1,2), \ D6(1,1),\ D7(0,3),\ D7(1,4),\ D8 (0,2),\ D8(1,3).\]
 \label{casesRW}
\end{lemma}
\begin{proof}
Consider case A1(6,5). By Riemann-Hurwitz's formula, the automorphism $\sigma_{14}$ acts on the curve $C_4\subset\Fix(\sigma_2)$ fixing four points, and it also acts on each of the five rational curves in $\Fix(\sigma_2)$, fixing two points on each. Therefore $\sigma_{14}$ fixes a total of 14 points. By a previous computation, the fixed locus $\Fix(\sigma_{14})$ consists of seven points. Therefore this case is not admissible.
A similar argument can be used to exclude the other cases.

\end{proof}

\begin{lemma}
Case C1(1,3) is not admissible.
\label{c1(13)}
\end{lemma}
\begin{proof}
Let $E$ be the elliptic curve fixed by the involution $\sigma_2=\sigma_{14}^7$. The curve $E$ is preserved by $\sigma_{14}$. Moreover, $E$ is not fixed by $\sigma_7$ pointwise but it is invariant for $\sigma_7$ because we are in Case C. Thus, since $E$ is elliptic, the automorphism $\sigma_{14}$ acts as a translation on $E$.
Let $\mathcal E$ be the elliptic fibration induced by $E$, with fiber $E$ over $t=0$. Since fixed curves do not meet, the zero section is not fixed by the involution $\sigma_2$. The involution fixes three rational curves since $k_2=3$ and they are contained in the fiber $F_\infty$ over $t=\infty$.
The only possible types of singular fibers that can contain three curves fixed by the involution are the types $I_6$, or $III^*$, or $I_4^*$. 
\begin{figure}[h]
\begin{minipage}{0.33\textwidth}
\centering
\begin{tikzpicture}[xscale=.6,yscale=.5, thick, every node/.style={circle, draw, fill=black!50, inner sep=0pt, minimum width=4pt}]
  \node (n2) at (-3,0) {};
  \node [double] (n3) at (-2,0) {};
  \node (n4) at (-1,0) {};
  \node [double](n5) at (0,0) {};
  \node (n6) at (1,0) {};
  \node [double](n7) at (2,0) {};
  \node (n8) at (3,0) {};
  \node (n9) at (0,-1) {};

  \foreach \from/\to in {n2/n3,n3/n4, n4/n5,n5/n6,n6/n7,n7/n8, n5/n9}
    \draw (\from) -- (\to);

\end{tikzpicture}
\caption{Fiber $III^*$} \label{fig:E7}\end{minipage}
\begin{minipage}{0.33\textwidth}
        \centering
\begin{tikzpicture}[xscale=.6,yscale=.5, thick, every node/.style={circle, draw, fill=black!50, inner sep=0pt, minimum width=4pt}]
  \node (n1) at (-3,1) {};
  \node (n2) at (-3,-1) {};
  \node [double] (n3) at (-2,0) {};
  \node (n4) at (-1,0) {};
  \node [double](n5) at (0,0) {};
  \node (n6) at (1,0) {};
  \node [double](n7) at (2,0) {};
  \node (n8) at (3,1) {};
  \node (n9) at (3,-1) {};

  \foreach \from/\to in {n1/n3,n2/n3,n3/n4, n4/n5,n5/n6,n6/n7,n7/n8, n7/n9}
    \draw (\from) -- (\to);
\end{tikzpicture}
\caption{Fiber $I_4^*$}\label{I_4^*}
    \end{minipage}\hfill
    \begin{minipage}{0.33\textwidth}
\centering
\begin{tikzpicture}[xscale=.6,yscale=.5, thick, every node/.style={circle, draw, fill=black!50, inner sep=0pt, minimum width=4pt}]
  
  \node (n2) at (-3,-1) {};
  \node [double] (n3) at (-2,0) {};
  \node (n4) at (-1,0) {};
  \node [double](n5) at (0,-1) {};
  \node [style=rectangle, draw, inner sep=0pt, minimum width=4pt,minimum height=4pt, fill=black!50,] (n6) at (-1,-2) {};
  \node [double](n7) at (-2,-2) {};
  
  \foreach \from/\to in {n2/n3,n3/n4, n4/n5,n5/n6,n6/n7,n7/n2}
    \draw (\from) -- (\to);
\end{tikzpicture}
\caption{Fiber $I_6$} \label{I_6}
\end{minipage}
\end{figure}

If $F_\infty$ is of type $III^*$, then the three curves which are fixed by $\sigma_{2}$ are represented by the three double circles in Figure \ref{fig:E7}.
The zero section would meet the external component of the fiber $III^*$ and thus it would be fixed by $\sigma_2$, which we already observed is impossible.
By a similar argument, we may exclude the case when $F_\infty$ is of type $I_4^*$, as shown in Figure \ref{I_4^*}.

Suppose that $F_\infty$ is of type $I_6$. By analyzing the types of points, it can be seen that one of the curves of the fiber $I_6$ which is not fixed by $\sigma_2$ must be fixed by $\sigma_7$. Such a curve is represented by a square in Figure \ref{I_6}. Since $\sigma_7$ must preserve the fiber, this is impossible. \end{proof}

\begin{lemma}
The following cases are not admissible: 
\[A2 (10,1),\ A2(3,1),\ C2(1,2),\ C2(0,1),\ C3(1,4),\ C3(0,3),\ D4 (3,1),\ D5(3,4), \ D6 (3,3).\]
\label{consecPts}
\end{lemma}
\begin{proof}
Observe that in Case A2(10,1), $\Fix(\sigma_2)=C_{10}\cup R$, where $C_{10}$ a curve of genus 10 and $R$ a rational curve, and neither of these curves are fixed by $\sigma_{14}$. The automorphism $\sigma_{14}$ fixes five isolated points, two of which lie on $R$. As observed in Remark \ref{types}, isolated points on a rational curve are of consequent types but this is in contradiction with the types of points for A2 (see Table \ref{tab:14}). 
The other cases can be excluded by a similar argument.
\end{proof}


\begin{lemma}
Suppose that the involution $\sigma_2$ fixes a curve $C_7$ of genus seven. Then the curve $C_7$ contains two fixed points by $\sigma_{14}$, which cannot be of the same type. 
\label{notequal}
\end{lemma}

\begin{proof}
First, note that $\sigma_{14}$ acts with order seven on $C_7$. Thus, by Riemann-Hurwitz it has exactly two fixed points, which we call $p$ and $q$.

Considering the line bundle $L$ associated to $8p$, by Riemman-Roch we have $h^0(C_7,L)\geq 2$ so that we obtain a finite (surjective) morphism $f: C_7\to\mathbb P^1$ of degree $d\leq 8$. Now, because $\sigma$ fixes $p$, $\sigma$ and $f$ induce an automorphism $\tilde{\sigma}$ (of order $7$) on $\mathbb{P}^1$. This automorphism has two fixed points, say $\tilde{p}$ and $\tilde{q}$, and we must have (up to relabeling) $f^{-1}(\tilde{p})=p$ and $f^{-1}(\tilde{q})=q$. Moreover, we can assume $\tilde{p}=(0:1)$ and $\tilde{q}=(1:0)$.

We can thus choose local coordinates $z$ on $\mathbb{P}^1$ centered on $\tilde{p}$ so that the action of $\tilde{\sigma}$ on $\tilde{p}$ is given by multiplication by $\zeta_{14}^{2j}$ and on $\tilde{q}$ it is given by multiplication by $\zeta_{14}^{14-2j}$ (for some $j$). Note that $1/z$ is then a local coordinate centered on $\tilde{q}$.
In fact we can choose local coordinates on $C_7$ which are compatible with the above so that $f$ is given by $z\mapsto z^{d}$ around $\tilde{p}$ (and analogously for $\tilde{q}$). Using this, we see that the local action of $\sigma$ on $p$ must be given by multiplication by $\zeta_{14}^{2j/d}$ and on $q$ it is given by multiplication by $\zeta_{14}^{(14-2j)/d}$. 

The local action of $\sigma$ on $p$ and $q$ as points in $X$ can thus be diagonalized so that $p$ is a point of type $A_{i,14}$ where $i= 2j/d -1$ or $2j/d$, and $q$ is a point of type $A_{j,14}$ where $k=(14-2j)/d-1$ or $(14-2j)/d$. In any case, $i\not\equiv k \mod 14$ so that $p$ and $q$ cannot be of the same type.
\end{proof}

As a consequence we can prove:

\begin{lemma}
Case C2(7,1) is not admissible.
\label{c2(71)}
\end{lemma}

\begin{proof}
According to Table \ref{tab:14}, in case C2 the automorphism $\sigma_{14}$ fixes exactly one point of type $A_{2,14}$, one point of type $A_{3,14}$, and two points of type $A_{6,14}$. Two of these are on the rational curve fixed by $\sigma_2$ and two are on the genus seven curve $C_7$ fixed by $\sigma_2$. Since the fixed points on $R$ must be of consecutive types (see Remark \ref{types}), the two points of type $A_{6,14}$ lie on $C_7$. This contradicts Lemma \ref{notequal}.
\end{proof}

Thanks to \cite{brandhorst2021}, we also prove:

\begin{lemma}
Cases $D1$ and $D7$ are not admissible.
\label{d1d7}
\end{lemma}

\begin{proof}
By \cite[Corollary 1.3]{brandhorst2021}, there are exactly 12 distinct deformation classes of K3 surfaces  $X$ carrying
a purely non-symplectic automorphism $\sigma$ of order $14$. In Section \ref{sec-ex}, we show  all 12 cases listed in Table \ref{tab:ex} indeed occur. Therefore, it suffices to observe the different cases determine different deformation classes.

In fact, looking at the eigenvalues of the induced isometry $\sigma^*$ on $H^2(X,\mathbb{Z})$ we see that different cases determine at least 11 deformation classes. With the exception of cases  $C1(6,1)$ and $C1(7,2)$, the different cases determine 11 distinct vectors $(d_{14},d_7,d_2,d_1)$ (see Table \ref{tab:ex}). So we analyze these two cases separately.

By \cite[Theorem 1.5.2]{dolQuad}, if $(X,\sigma)$ is of type $C1(6,1)$ and $(\tilde{X},\tilde{\sigma})$ is of type $C1(7,2)$, then the invariant lattices $S(\sigma^7)$ and $S(\tilde{\sigma}^7)$ do not lie in the same genus. And, since the deformation class of a pair $(X,\sigma)$ is determined by the collection of genera of the lattices $S(\sigma^j)$ by \cite[Theorem 1.4]{brandhorst2021}, we conclude these two cases indeed determine two distinct deformation classes.
\end{proof}

Using Table \ref{tab:14larga} and combining Lemmas \ref{casesRW}, \ref{c1(13)}, \ref{consecPts}, \ref{c2(71)} and \ref{d1d7} we have thus proved:

\begin{proposition}
Let $\sigma_{14}$ be a purely non-symplectic automorphism on a $K3$ surface. Then the admissible cases according to the possible fixed locus are listed in Table \ref{tab:ex}.
\end{proposition}

\begin{table}[H]\centering
\begin{tabular}{c!{\vrule width 1.5pt}c|c|c|c|c|c|cc}
&$N$&$\alpha_{14}$&$\chi_{14}$&$\chi_7$&$\chi_2$&$(g_2,k_2)$&$(d_{14},d_7,d_2,d_1)$&\\
\noalign{\hrule height 1.5pt}
A1&7&0&7&3&-14&$(9,1)$&(3,0,1,3)\\
&&&&&0&$(3,2)$&(2,1,0,4)\\
\hline
A2&5&0&5&3&-16&$(9,0)$&(3,0,2,2)\\

\noalign{\hrule height 1.5pt}
B3&12&1&14&10&0&$(6,5)$&(2,0,0,10)\\
\noalign{\hrule height 1.5pt}
C1&6&0&6&10&-8&$ (6,1),(7,2)$&(2,0,4,6)\\
&&&&&6&$(0,2)$&(1,1,3,7)&\\\hline
C2&4&0&4&10&-10&$(6,0)$&(2,0,5,5)\\
\hline
C3&6&1&8&10&-6&$(6,2)$&(2,0,3,7)\\
\noalign{\hrule height 1.5pt}
D2&3&0&3&17&-4&$(3,0)$&(1,0,8,8)\\\hline
D3&7&0&7&17&0&$(3,2)$&(1,0,6,10)\\\hline
D8&11&1&13&17&6&$(3,5)$&(1,0,3,13)\\
\end{tabular}

\caption{}\label{tab:ex}
\end{table}

\subsection{Realization by examples}\label{sec-ex}

It remains to show each case listed in Table \ref{tab:ex} is indeed realizable. For each possibility, we construct explicit examples of K3 surfaces carrying a purely non-symplectic automorphism $\sigma_{14}$ (of order $14$) that has the desired type of fixed locus.


\begin{example}\label{example_A1_(9,1)}{\bf (Case A1(9,1))} 
Consider $(X_{a,b}, \sigma_{14})$, taking $X_{a, b}$ to be the elliptic K3 surface with Weierstrass equation
\begin{equation*}
    y^2=x^3+(at^7+b)x+(t^7-1), \quad t\in\mathbb P^1
\label{eq_fibr}
\end{equation*}
where $a,b \in \mathbb{C}$, as in \cite[Example 6.1]{ArtebaniSartiTaki}, 
and letting $\sigma_{14}$ be the purely non-symplectic order 14 automorphism:
\[
\sigma_{14}:(x,y,t) \mapsto (x,-y,\zeta_7^4 t)
\]
where $\zeta_7$ denotes a primitive $7$-th root of unity.

If $a$ and $b$ are generic, then $X_{a,b}$ contains a fiber of type $III$ at $t=(1:0)$ and $21$ singular fibers of type $I_1$. One can show that the fixed locus of $\sigma_{14}$ is such that $m=( 0 , 0 , 0 , 0 , 1 , 2 ,4)$. In fact it is of type $A1(9,1)$. It can be described as follows: the four isolated points of type $A_{6,14}$ lie on a curve which is fixed by $\sigma_7$, namely the fiber at $t=(0:1)$; the other three points lie on the fiber of type $III$: the tangency point, along with one other point on each of the two components. Moreover, the involution $\sigma_2$ fixes the zero section (which is rational) and the trisection (which has genus 9).


\end{example}

\begin{example}{\bf(Case A1(3,2))}\label{example_A1_(3,2)}
Consider $(X, \sigma_{14})$, where $X$ is the  elliptic K3 surface with Weierstrass equation given by
\[y^2=x(x^2+(t^7+1)),\ t\in\mathbb P^1,\] and 
$\sigma_{14}\colon(x,y,t)\mapsto (x,-y,\zeta^4_{7}t)$
is a purely non-symplectic automorphism of order $14$. We note that $X$ contains eight singular fibers of type $III$. The fixed locus of $\sigma_7$ is given by an elliptic curve at $t=(0:1)$ and three points that lie on the fiber of type $III$ at $t=(1,0)$. On the fiber of type $III$, one of the three points is the tangency point, while the remaining two lie on different components. Therefore, we are in case $A$ (for $\sigma_7$). 
 
The fixed locus of $\sigma_{14}$ is such that $m=( 0 , 0 , 0 , 0 , 1 , 2 ,4)$ and in fact we can check it is of type $A1(3,2)$. On the elliptic curve $\sigma_{14}$ acts as an involution and we obtain $4$ fixed points there, the other $3$  fixed points are again in the fiber of type $III$ at $t=(1,0)$ distributed as above.   
The involution $\sigma_2$ fixes the bisection which has genus $3$, and two rational curves: the zero section  and the two torsion section given by $x=y=0$. Therefore, $(g_2,k_2)=(3,2).$
\end{example}

\begin{example}{\bf(Case A2(9,0))}\label{example_A2_(9,0)} Let us consider $(X, \sigma_{14})$ the elliptic K3 surface $X$ together with the automorphism $\sigma=\sigma_{14}$ from Example \ref{example_A1_(3,2)}. 

The translation $\tau$ given by $(x,y,t)\mapsto ((y/x)^2-x,(y/x)^3-y,t)$ (which is the translation by the $2-$torsion section) is a symplectic involution that commutes with $\sigma$. As a consequence, the composition $\sigma'\coloneqq \sigma\circ \tau$ is also a purely non-symplectic automorphism of order $14.$ We remain in case $A$ for $\sigma_7$ and the fixed locus of $\sigma'$ is  such that $m=( 0 , 0 , 1 , 0 , 0 , 0 , 4)$. Indeed, $\sigma'$ acts as an involution on the elliptic curve $E$ at $t=(0,1)$ and $E$ contains four fixed points. Due to the fact that $\tau$ has only eight fixed points, which are precisely the tangency points on the singular fibers of type $III,$ we only have one additional fixed point lying on the fiber at $t=(1,0).$ The involution does not fix any rational curves and therefore we are in case $(g_2,k_2)=(9,0).$ We note that this case is also presented in \cite[Section 7.2, p.19]{GP}.
\end{example}

\begin{example}\label{example_B3}{\bf (Case B3)}
Consider $(X_{a,b}, \sigma_{14})$, where we let $X_{a, b}$ be the elliptic K3 surface in Example \ref{example_A1_(9,1)} with $a=0$. $X_{0,b}$ contains a fiber of type $II^*$ at $t=(1:0)$, a smooth fiber at $(0:1)$, and $14$ singular fibers of type $I_1$. With the order 14 automorphism $\sigma_{14}$ given in \eqref{eq_fibr}, the component of multiplicity $6$ on the $II^*$ fiber is fixed by $\sigma$ and the action on the fiber over $t=(0:1)$ is an involution, so it has 4 fixed points. Checking types of fixed points, we find $m=(3, 2 , 1, 1, 1 , 4)$ with $\alpha_{14}=1$. 
\end{example}

\begin{example}{\bf (Case C1(6,1))}\label{example_C1_(6,1)}
Consider $(X_{a, b}, \sigma_{14})$ from Example \ref{example_A1_(9,1)}, with $a$ generic and $b$ such that $b^3=-\frac{27}{4}$. Then $X_{a,b}$ contains a fiber of type $III$ at $t=(1:0)$, a fiber of type $I_7$ at $t=(0:1)$ and $14$ singular fibers of type $I_1$. In this case the fixed locus of $\sigma_{14}$ is such that $m=( 0, 0, 0 , 1 , 1, 2 , 2)$. 
The trisection $\{y=0\}$ is a curve of genus 6 and it is fixed by the involution, as well as the zero section. Thus the invariants of the fixed locus of the involution $\sigma_2$ are $(g_2,k_2)=(6,1)$.
\end{example}

\begin{example}{\bf(Case C1(7,2))}\label{example_C1_(7,2)}
Let $(X, \sigma_{14})$, the elliptic K3 surface with Weierstrass equation
\[y^2=x^3+4t^4(t^7-1), \ t\in\mathbb P^1,\]
together with the order 14 purely non-symplectic automorphism $\sigma_{14}$ given by 
\[\sigma_{14}(x,y,t)=(\zeta^4_7x, -\zeta_7^6y,\zeta_7^3t).\]
We note that the singular fibers are of type $IV^*$ over $t=(0:1)$ and type $II$ over $t=(1:0)$, in addition to seven fibers of type $II$. The square of $\sigma_{14}$ fixes the  component of multiplicity 3 on the fiber of type $IV^*$, so this example falls under case C. The involution $\sigma_2$ acts as a reflection on this fiber, and so the fixed locus $\Fix(\sigma_{14})$ only contains points. The 3-section $\{y=0\}$ has genus seven and it is fixed by the involution, as well as the zero section and one rational component of the fiber $IV^*$. Thus the invariant of the fixed locus of the involution are $(g_2,k_2)=(7,2)$.
This surface appears in \cite[Table 3]{Brandhorst}, with a non-symplectic automorphism of a different order.
\end{example}

\begin{example}{\bf(Case C1(0,2))}\label{example_C1_(0,2)}
Let us consider $(X, \sigma_{14})$, the  elliptic K3 surface $X$ with Weierstrass equation given by 
\[y^2=x^3+t^2x+t^{10},\ t\in\mathbb P^1, \]
and the order 14 purely non-symplectic automorphism $\sigma_{14}\colon(x,y,t)\mapsto (\zeta_{7}x,\zeta^5_{7}y,-\zeta_{7}t)$.
Note that $X$ contains a fiber of type $IV$ at $t=(1:0),$ a fiber of type $I^*_0$ at $t=(0:1)$ and $14$ singular fibers of type $I_1.$
The fixed locus of $\sigma_7$ fixes one rational curve, the non-reduced component, and eight points, so this example falls under Case C. Because the involution $\sigma_2$ fixes only three rational curves, we see that $\sigma_{14}$ is of type $C1$ with $(g_2,k_2)=(0,2).$
\end{example}

\begin{example}{\bf (Case C2(6,0))} \label{example_C2(6,0)}
Consider $(X, \sigma_{14})$, where $X$ is the K3 surface with equation \[y^2=x^7s-t^2(t-s^2)(t-2s^2)\] in $\mathbb{P}(4,2,1,1)_{(y,t,x,s)}$ and $\sigma_{14}\colon (y,t,x,s) \mapsto (-y,t,\zeta_{7}x,s)$ is a purely non-symplectic automorphism of order $14$. One can see that the points $(-1:1:0,0)$ and $(1:1:0,0)$ are of type $A_{1}.$ Moreover, at the point $(0:0:0:1)$ we have a singularity of type $A_{6}.$ Since $\sigma_7$ fixes the rational curve $C_{x}\coloneqq \{x=0\}$ and eight points, this example falls under Case $C.$ The only curve fixed by the involution $\sigma_2$ is $C_{y}$, which has genus six.
\end{example} 

\begin{example}{\bf (Case C3)} \label{example_C3_(6,2)}
Consider $(X, \sigma_{14})$, the elliptic surface $X$ with Weierstrass equation
\[y^2=x^3+t^3(t^7+1),\ t\in\mathbb P^1,\]
and the order 14 purely non-symplectic automorphism $\sigma_{14}$
\[\sigma_{14}(x,y,t)=(\zeta_7^2x, -\zeta_7^3y, \zeta_7^3t).\]
The singular fibers consist of a type $I_0^*$ fiber over $t=(0:1)$, a type $IV$ fiber over $t=(1:0)$, and seven type $II$ fibers (cusps). We call $R$ the non-reduced component of the $I_0^*$ fiber. The involution $\sigma_2$ fixes the zero section, the rational curve $R$ and the 3-section $C$ given by $y=0$. The curve $C$ passes through the center of the $IV$ fiber and through the cusps, and so $C$ has genus six by Riemann-Hurwitz. Thus the invariants of the involution are $(g_2,k_2)=(6,2)$ which corresponds to Case C3. 
The fixed locus of $\sigma_{14}$ consists of the curve $R$ and six points. 


Another example for Case $C3$ is given as follows. Let $X$ be the K3 surface with equation \[x^2+y^3z+z^7+w^{14}=0\] and weights $ (7,4,2,1)$. Singularities can occur only at singularities of $\mathbb P(7,4,2,1)$ and
one can see that the point $(0:1:0:0)$ is an $A_3$ singularity and $(0:\zeta_6^j:1:0), j=1,3,5$ is an $A_1$ singularity.

\begin{figure}[ht]
\centering
\begin{tikzpicture}[xscale=.6,yscale=.5, thick]

\draw [thick] (0,8)--(6,8);
\node [right] at (6,8){$C_{w}$};

\draw (1,8.5)--(1,4.5);
\draw (1.5,8.5)--(1.5,4.5);
\draw (2,8.5)--(2,4.5);

\draw (0,5)--(6,5);
\node [right] at (6,5){$C_{x}$};

\draw (5,8.5)--(4,6.5);
\draw (4,7.5)--(5,5.5);
\draw (5,6.5)--(4,4.5);

\end{tikzpicture}
\label{fig_weighted}
\caption{C3}
\end{figure}

After resolving the singularities, the curve $C_w:=\{w=0\}$ has genus zero, while the transform of $C_x:=\{x=0\}$ has genus six. 
The automorphism $\sigma_{14}: (x,y,z,w)\mapsto(x,y,z,\zeta_{14}w)$ is a purely non-symplectic automorphism of order 14 and it fixes the rational curve $C_w$. Its square $\sigma_7$ fixes $C_w$ as well, so that this example falls under Case C. Moreover, the involution $\sigma_2$ fixes $C_w$ and $C_x$ and the central fiber of the resolution of the $A_3$ (another rational curve). Therefore, $\sigma_{14}$ is of type C3 and the invariants of the involution are $(g_{2}, k_{2}) = (6,2)$.
\end{example}

\begin{example}{\bf (Case D2)} \label{example_D2}
Let $(X, \sigma_{14})$ be the K3 surface $X$ with equation \[y^2=x^7s-t^2(t-s^2)^2\] in $\mathbb{P}(4,2,1,1)_{(y,t,x,s)}$ and the order 14 purely non-symplectic automorphism $\sigma\colon (y,t,x,s)\mapsto (-y,t,\zeta_{7}x,s)$.

The points $(-1:1:0,0)$ and $(1:1:0,0)$ are of type $A_{1}.$ Moreover, at the points $(0:0:0:1)$ and $(0:1:0:1),$ we have singularities of type $A_{6}.$ Since $\sigma_7$ fixes two rational curves $C_1$ and $C_2$, appearing when $x=0,$ this example falls under Case $D.$ The only curve fixed by the involution $\sigma_2$ is $C_{y}$, which has genus three. See also \cite[Section 7.3]{GP}.
\end{example}

\begin{example}{\bf (Case D3)}\label{example_D3}
Consider $(X, \sigma_{14})$, where $X$ is the K3 surface with equation \[x^2=w^7y+y^4+z^7\] in $\mathbb{P}(14,7,4,3)_{(x,y,z,w)}$ given in \cite{ABS}, and $\sigma_{14}$ the order 14 purely non-symplectic automorphism $\sigma_{14} \colon (x,y,z,w)\mapsto (-x,y,z,\zeta_{7}w)$.

We have the following: point $(1:0:1:0)$ of type $A_{1}$; points $(1: 1:0:0)$ and $(-1: 1:0:0)$, both of type $A_{6}$; and point $(0:0:0:1)$ of type $A_2$ (Figure \ref{fig:D3}).

\begin{figure}[h!]
\centering
    
\begin{tikzpicture}[line cap=round,line join=round,>=triangle 45,x=0.6cm,y=0.6cm]
\clip(-1.,2.) rectangle (9.5,9.);
\draw [line width=1.5pt] (0.,8.)-- (8.,8.);
\draw [line width=1.5pt] (0.,3.)-- (8.,3.);
\draw [line width=1.pt] (6.66,8.5)-- (5.36,7.02);
\draw [line width=1.pt] (5.46,7.6)-- (6.4,6.);
\draw [line width=1.pt] (5.38,5.38)-- (6.66,6.66);
\draw [line width=1.pt] (5.449316303531188,5.890578512396699)-- (6.33586776859505,4.418001502629605);
\draw [line width=1.pt] (5.329105935386936,3.8169496619083425)-- (6.591314800901587,4.94392186326071);
\draw [line width=1.pt] (5.344132231404967,4.3729226145755105)-- (6.275762584522924,2.4345304282494387);
\draw [line width=1.pt] (4.796061012223987,8.487630817070546)-- (3.4960610122239872,7.007630817070546);
\draw [line width=1.pt] (3.596061012223987,7.587630817070546)-- (4.536061012223987,5.987630817070547);
\draw [line width=1.pt] (3.516061012223987,5.3676308170705465)-- (4.796061012223987,6.647630817070547);
\draw [line width=1.pt] (3.585377315755175,5.8782093294672455)-- (4.471928780819037,4.4056323197001515);
\draw [line width=1.pt] (3.465166947610923,3.804580478978889)-- (4.727375813125574,4.931552680331256);
\draw [line width=1.pt] (3.480193243628954,4.360553431646057)-- (4.411823596746911,2.4221612453199852);
\draw [line width=1.pt] (2,8.5)-- (2,6.5);
\draw [line width=1.pt] (2,4.5)-- (2,2.5);
\draw (8.131725487029298,8.351681812580267) node[anchor=north west] {$C_{w}$};
\draw (8.264237316430284,3.428357689451359) node[anchor=north west] {$C_{z}$};
\draw (-0.7,6.710573771537298) node[anchor=north west] {$C_{x}$};
\draw [line width=1.pt] (2.2,4.3)-- (0.3,4.3);
\draw [line width=1.5pt] (0.7989755342900065,8.690920310306069)-- (0.7989755342900065,2);
\begin{scriptsize}
\
\end{scriptsize}
\end{tikzpicture}
\caption{D3}\label{fig:D3}

\end{figure}

Since $\sigma_7$ fixes the rational curves $C_{z}\coloneqq \{z=0\}$ and $C_{w}\coloneqq \{w=0\},$ we are in case $D.$ 
The involution $\sigma_2$ fixes the curve $C_{x}\coloneqq \{x=0\}$ of genus $3$ and two rational curves given by the component of the $A_1$ and one of the components of the $A_2.$ The involution $\sigma_2$ also exchanges the two $A_{6}$ points. 
\end{example}

\begin{example}{\bf (Case D8)} \label{example_D8}
Again, consider $(X_{a, b}, \sigma_{14})$, the K3 surface together with the automorphism from Example \ref{example_A1_(9,1)}. If $a=0$ and $b$ is such that $b^3=-\frac{27}{4}$, it follows that $X_{a,b}$ contains a fiber of type $II^*$ at $t=(1:0)$, a fiber of type $I_7$ at $t=(0:1)$, and seven singular fibers of type $I_1$. 
The fixed locus of $\sigma_{14}$ is of type D8. 


Observe that the surface $y^2=x^3+t^3x+t^8,$\ $t\in\mathbb P^1$ given in \cite[ Example 7.5]{kondo_trivial} admits the purely non-symplectic order 14 automorphism
$
\sigma_{14}: (x,y,t) \mapsto (\zeta_7^3x,-\zeta_7y,\zeta_7^2t)
$ and corresponds to case D8 as well.
\end{example}

\section{Order 21}
\label{sec-order21}

Purely non-symplectic automorphisms of order $21$ on K3 surfaces have been classified in \cite{Brandhorst}. Here we present a new proof and a more detailed description of Brandhorst's result. More precisely, using the same kind of approach from the previous section, we show that the examples of \cite[Table 3]{Brandhorst} fit the invariants of Table \ref{tab:21} below, and we prove:


\begin{proposition}

\label{thm21}

The fixed locus of a non-symplectic automorphism of order $21$ on a K3 surface is not empty and it consists of either:

\begin{enumerate}[(i)]
    \item The union of $N_{21}$ isolated points, where $N_{21}\in \{4,7\}$; or
    \item The disjoint union of a rational curve and $N_{21}$ isolated points, where $N_{21}\in\{8,11\}$.
\end{enumerate}
Moreover, all these possibilities occur, and a more detailed description is given in Table \ref{tab:21} below, where $\sigma_7\doteq \sigma_{21}^3$ and $\sigma_3\doteq \sigma_{21}^7$.

\begin{table}[H]
\begin{tabular} {c!{\vrule width 1.5pt} c|c|c|c}
&  $\Fix(\sigma_{21})$ & $\Fix(\sigma_{7})$ & $\Fix(\sigma_3)$ &Example\\
\noalign{\hrule height 1.5pt}
  C(3,2,3)  & $R \sqcup\{p_1,\ldots,p_8\}$ &$R\sqcup \{p_1,\ldots,p_8\}$ & $C_3\sqcup R\sqcup R'\sqcup\{p_1,p_2,p_3\}$ &\ref{ex21-C1}\\
 
  C(3,1,2) &$\{p_1,\ldots,p_7\}$ &$R\sqcup \{p_1,\ldots,p_5,q_1,q_2,q_3\}$& $C_3\sqcup R\ \sqcup\{p_1,q_1\}$ &\ref{ex21-C2}\\
  C(3,0,1) & $\{p_1,\ldots,p_4\}$ & $R\sqcup \{p_1,\ldots,p_8\}$& $C_3\sqcup\{p_1\}$ &\ref{ex21-C3}\\ \noalign{\hrule height 1.5pt}
 B(3,3,4) &$R \sqcup \{p_1,\ldots,p_{11}\}$ &$E\sqcup R\sqcup \{p_1,\ldots,p_8\}$ & $C_3\sqcup R\sqcup R'\sqcup R''\sqcup \{p_1,\ldots,p_4\}$&\ref{ex21-B4}
\end{tabular}
\caption{Order 21}
\label{tab:21}
\end{table}
\end{proposition}

In order to prove Proposition \ref{thm21}, we first note that, as we observed in Section \ref{sec:background}, at any fixed point a purely non-symplectic automorphism $\sigma_{21}$ of order $21$  acts as multiplication by the matrix $A_{i,21}$ for some $i$, with \[A_{i,21}\coloneqq \begin{pmatrix} \zeta_{21}^{1+i} & 0 \\ 0 & \zeta_{21}^{21-i} \end{pmatrix},\ 0\leq i \leq 10.\]

Thus, the holomorphic Lefschetz  formula \eqref{eq-Lefhol} applied to $\sigma_{21}$ gives us the following linear system of equations:

\begin{equation}
\begin{cases}
3 m_{6,21} &= 3 + 4 m_{1,21} - 5 m_{2,21} - 4 m_{4,21} + 8 m_{5,21} \\
 3 m_{7,21} &= 3 - 5 m_{1,21} + 4 m_{2,21} - 13 m_{4,21} + 17 m_{5,21} \\
 m_{8,21} &= 1 - 2 m_{1,21} + 2 m_{2,21} - 5 m_{4,21} + 6 m_{5,21} \\
 m_{9,21} &= 3 - 4 m_{1,21} + 4 m_{2,21} - 3 m_{3,21} - 3 m_{4,21} + 7 m_{5,21} \\
 2 m_{10,21} &= 2 - 3 m_{1,21} + 3 m_{2,21} - 2 m_{3,21} - 3 m_{4,21} + 6 m_{5,21} \\
 6 \alpha_{21} &= m_{1,21} + m_{2,21} - m_{4,21} + 2 m_{5,21} 
\end{cases}\end{equation}
where $\alpha_{21}\coloneqq \sum 1-g(C)$ and the sum is taken over all curves $C$ fixed by $\sigma_{21}$.

Moreover, considering the non-symplectic automorphism $\sigma_7=\sigma_{21}^3$ of order 7, we know that 
\[
\begin{cases}
m_{1,21}+m_{5,21}+m_{8,21} &\leq m_{1,7} \\
m_{2,21}+m_{4,21}+m_{9,21} &\leq m_{2,7} \\
m_{3,21}+m_{10,21} &\leq m_{3,7} 
\end{cases}
\]

We note also that points of type $A_{6,21}$ and $A_{7,21}$ lie on a curve fixed by $\sigma_7$ (but not fixed by $\sigma_{21}$) and points of type $A_{j,21}$, where $j=2,3,5,6,8,9$, lie on a curve fixed by $\sigma_3=\sigma_{21}^7$ (but not fixed by $\sigma_{21}$). For this reason, we choose $r\coloneqq m_{6,21}+m_{7,21}$. Using MAGMA, we obtain the following four possibilities for the vector $(m_{1,21},\ldots,m_{10,21};\alpha_{21},r)$:
\begin{align*}
v_1=(3,3,1,0,0,0,0,1,0,0;1,0) && v_2=(0,0,0,0,0,1,1,1,3,1;0,2) \\
v_3=(0,0,1,0,0,1,1,1,0,0;0,2) && v_4=(3,2,1,1,1,3,0,0,0,0;1,3)
\end{align*}

Furthermore, we observe the following:

\begin{lemma}
If the fixed locus of $\sigma_{21}$ is described by one of the vectors $v_1,v_2,v_3$, then the fixed locus of $\sigma_7=\sigma_{21}^3$ is as in Case C of Table \ref{tab:7}.
If it is described by the vector $v_4$, then the fixed locus of $\sigma_7$ is as in Case B.
\end{lemma}
\begin{proof}
We first observe that $\sigma_7$ cannot be of type A. Assume we are in Case A. We know that $\Fix(\sigma_{21}) \subseteq \Fix(\sigma_{7})$. By the Riemann-Hurwitz formula, the genus one curve in $\Fix(\sigma_7)$ would contain either none or three isolated points fixed by $\sigma_{21}$, and thus $r=0$ or $3$. But the cases with these values of $r$ both have $\alpha_{21}=1$, which is not possible in Case A (recall that in Case A, a fixed curve must have genus 1, as shown in Table \ref{tab:7}).

Case D for $\sigma_7$ is not admissible either. In fact, if $\sigma_7$ is as in case D, then $\Fix(\sigma_7)$ contains two rational curves. If they were both pointwise fixed by $\sigma_{21}$, this would give $\alpha_{21}=2$.
If one curve is pointwise fixed and the other one is invariant, then $\alpha_{21}=1$ and $r=2$.
If both curves are invariant but not pointwise fixed, then $\alpha_{21}=0$ and $r=4$. These cases do not appear among the admissible ones. Therefore we conclude that $\sigma_7$ must fall under Case $B$ or Case $C$.

We now observe that the situation described by the vector $v_4$ is only possible in Case B: since $r=3$ in this case, it means that there are three points on curves fixed by $\sigma_7$ and they are not fixed by $\sigma_{21}$. Thus there must be an elliptic curve in $\Fix(\sigma_7)$. As we observed in Lemma \ref{prop:B}, if $\sigma_7$ fixes an elliptic curve and a rational curve as in Case B, the surface admits an elliptic fibration with a fiber of type $II^*$ and 14 fibers of type $I_1$. Since the fiber of type $II^*$ does not admit a symmetry of order three, $\sigma_{21}$ fixes the central curve of this fiber and eight points that lie on it. 

As for vector $v_2$ (respectively $v_3$), the fixed locus of $\sigma_{21}$ consists of seven (respectively four) points. Thus $\sigma_7$ cannot belong to Case B, since by the previous remark, it would fix too many points.

Assume now that we are in Case B and the vector $v_1$ describes the action of $\sigma_{21}$. Then the fixed locus of $\sigma_{21}$ is the union of a rational curve and eight points; since $r=0$, the action of $\sigma_{21}$ on the elliptic curve in $\Fix(\sigma_7)$ is a translation. But then the action should be a translation on the fiber $II^*$, and this is not the case.
\end{proof}

At last, we are now in position to prove Proposition \ref{thm21}:

\begin{proof}[Proof of Proposition \ref{thm21}]
Consider the induced action of $\sigma_{21}$ on $H^2(X,\mathbb{R})$ and recall the definition of $d_i\coloneqq \dim\  H^2(X,\mathbb{R})_{\zeta_{i}}$ for $i=1,3,7,21$.

For each $i=3,7,21$ we let $\chi_i$ denote the Euler characteristic of the fixed locus of $\sigma_i=\sigma^{\frac{21}i}$. By applying the topological Lefschetz formula \eqref{eq-Leftop} to $\sigma_{21},\sigma_7$ and $\sigma_3$, we obtain:
\begin{equation}\label{top21}\begin{cases}
\chi_{21} &= 2 + d_{21} - d_7 - d_3 + d_1 \\
\chi_7 &= 2 - (2 d_{21} + d_7) + 2 d_3 + d_1 \\
\chi_3 &= 2 - 6 d_{21} + 6 d_7 - d_3 + d_1
\end{cases}\end{equation}
Moreover, we know that 
\[
22 = \text{dim }H^2(X,\mathbb{R})= 12 d_{21} + 6 d_7 + 2 d_3 + d_1.
\]
Combining these equations one gets the possibilities given in Table \ref{tab211}.
\begin{table}[h]
\begin{tabular} {c|c|c|c|c|c|c}
 Type $\sigma_7$ &$\chi_7$ & $\chi_{21}$ & $\chi_3$ & $(d_{21},d_7,d_3,d_1)$ & $(m_{1,21},\ldots,m_{10,21} ,\alpha_{21},r)$ & $\Fix(\sigma_{21})$ \\
\hline
   & 10 & 10 & 3 & (1,0,1,8) & (3,3,1,0,0,0,0,1,0,0,1,0) & $R \sqcup\ $ 8 pts \\
  C& 10 & 7 & 0 & (1,0,2,6) & (0,0,0,0,0,1,1,1,3,1,0,2)  & 7 pts \\
  & 10 & 4 & -3 & (1,0,3,4) & (0,0,1,0,0,1,1,1,0,0,0,2)  & 4 pts \\ \hline
 B & 10&13 & 6&(1,0,0,10)& (3,2,1,1,1,3,0,0,0,0,1,3)&$R \sqcup\ $11pts
 \end{tabular} \caption{}\label{tab211}
 \end{table}
 
Thus, it remains to look at the fixed locus of $\sigma_3$, which by \cite{ArtebaniSarti} consists of $N_3$ isolated fixed points, a curve of genus $g_3\geq 0$, and $k_3$ rational curves, where by \cite[Theorem 2.2]{ArtebaniSarti} the following relation holds:
\[1-g_3+k_3=N_3-3.\]
In particular, $\chi_3=N_3+2(1-g_3+k_3)=3N_3-6$, and we can  list the possibilities for $(g_3,k_3,N_3)$ according to the value of $N_3$. 

If $\chi_3=3$, then $N_3=3$ and  by \cite{ArtebaniSarti} we have the following possibilities for the invariants $(g_3,k_3,N_3)$ of $\Fix(\sigma_3)$:
\begin{align*}(g_3,k_3,N_3)=(-,-,3),
(1, 0, 3),
(2, 1, 3),
(3, 2, 3).
\end{align*}
Similarly,
\begin{itemize}
    \item if $\chi_3=0$, then $N_3=2$ and the possibilities are
$(g_3,k_3,N_3)=(2, 0, 2), (3, 1, 2),  (4, 2, 2)$.

\item  if $\chi_3=-3$ then $N_3=1$ and the possibilities are
$ (g_3,k_3,N_3)=
(3, 0, 1),
(4, 1, 1)$.

\item if $\chi_3=6$, then $N_3=4$ and the possibilities are: $(g_3,k_3,N_3)=(3,3,4), (2,2,4), (1,1,4), (0,0,4)$.
\end{itemize}

Next, we observe that we can actually eliminate most of these possibilities.

As in Lemma \ref{homma}, the automorphism $\sigma_{21}$ acts with order seven on $\Fix(\sigma_3)$, and thus $C_{g_3}$ should admit an automorphism of order seven.  But if $g_3\geq 2$ and if $\phi$ is an automorphism of prime order $p$, we must have $p\leq 2g_3+1$. Then we may eliminate the case where $g_3 = 2$.

A curve of genus four does not admit an automorphism of order seven by \cite{database}, and thus $g_3\neq 4$.

Finally, if $\chi_3=3$, then $\Fix(\sigma_{21})$ consists of a fixed rational curve plus eight points. Since $\Fix(\sigma_{21})\subseteq \Fix(\sigma_3)$, using Riemann-Hurwitz we can also eliminate the triples $(g_3,k_3,N_3)=(-,-,3),
(1, 0, 3)$. 
The argument is similar for triples $(g_3,k_3,N_3)=(1,1,4),(0,0,4)$ with $\chi_3=6$.

Therefore, the possible cases are the ones listed in Table \ref{tab:21}.
\end{proof}

\begin{remark}\label{d21}
Note that in the proof of Proposition \ref{thm21}, we have \begin{equation*}
\rk S(\sigma_{21})=d_1,\ \rk S(\sigma_7)=2d_3+d_1,\ \rk S(\sigma_3)=6d_7+d_1.
\end{equation*}\end{remark}

We end this section by showing the examples in \cite{Brandhorst} are indeed compatible with the invariants of Table \ref{tab:21}, as claimed.

\begin{example}{\bf(Case C(3,2,3))}\label{ex21-C1} 
Let $(X,\sigma_{21})$ be the following elliptic K3 surface with the non-symplectic automorphism $\sigma_{21 }$ of order 21:
\[y^2=x^3+4t^4(t^7-1), t\in\mathbb P^1\quad \sigma_{21}: (x,y,t)\mapsto (\zeta_7^6\zeta_3 x, \zeta_7^2y, \zeta_7t).\]

The collection of singular fibers of the elliptic fibration consist of a fiber of type $IV^*$ over $t=0$, a fiber of type $II$ over $t=\infty$, and $7$ of type $II$ over the zeros of $t^7-1$. The fixed locus of $\sigma_7$ consist of the central component $R$ of the fiber $IV^*$, six isolated points on the fiber $IV^*$ and two points on the fiber $II$ over $t=\infty$. The automorphism $\sigma_{21}$ has the same fixed locus as $\sigma_7$. The fixed locus of $\sigma_3$ consists of the zero section, the curve $R$ and the 3-section $y=0$, which has genus three and 3 additional points. 

In particular, the invariants of $\Fix(\sigma_{21}^j),j=1,3,7$ are as in the first row of Table \ref{tab:21}.
\end{example}

\begin{example}\label{ex21-C2} {\bf(Case C(3,1,2))} 
Let $(X,\sigma_{21})$ be the following elliptic K3 surface with the non-symplectic automorphism $\sigma_{21}$ of order 21: 
\[y^2=x^3+t^3(t^7+1), t\in\mathbb P^1\quad \sigma_{21}: (x,y,t)\mapsto (\zeta_7^3\zeta_3 x, \zeta_7y, \zeta_7^3t).\]
The singular fibers of the elliptic fibration are $I_0^*+IV+7II$. The fixed locus $\Fix(\sigma_7)$ consists of the central component $R$ of the fiber $I_0^*$, four points on $I_0^*$, and four points on $IV$. The automorphism $\sigma_{21}$ does not fix $R$ and only fixes isolated points. The automorphism $\sigma_3$ exchanges three of the non-central components of the fiber $I_0^*$ and acts on the remaining one, and thus $(g_3,k_3,N_3)=(3,1,2).$

The conclusion is that the invariants of $\Fix(\sigma_{21}^j),j=1,3,7$ are as in the second row of Table \ref{tab:21}.
\end{example}

\begin{example} {\bf(Case C(3,0,1))}\label{ex21-C3} 
Let $X$ be the K3 surface whose equation in $\mathbb P^3$ is \[x_0^3x_1+x_1^3x_2+x_0x_2^3+x_0x_3^3=0.\]
This surface admits the purely non-symplectic automorphism of order 21 \[\sigma_{21}:(x_0,x_1,x_2,x_3)\mapsto (\zeta_7x_0, \zeta_7^5x_1, x_2,\zeta_3x_3)\] whose fixed locus consists of the four standard coordinate points. 
The fixed locus of $\sigma_3$ consists of the genus three curve $\{x_3=0\}\cap X$ and the point $p_1=(0:0:0:1)$. 

In particular, we see that the invariants of $\Fix(\sigma_{21}^j),j=1,3,7$ are as described in the third row of Table \ref{tab:21}.
\end{example}

\begin{example} \label{ex21-B4}{\bf(Case B(3,3,4))} 
Let $(X,\sigma_{21})$ be the following elliptic K3 surface with the non-symplectic automorphism $\sigma_{21}$ of order 21:
\[y^2=x^3+t^5(t^7-1), t\in\mathbb P^1\quad \sigma_{21}: (x,y,t)\mapsto (\zeta_{21}^2 x, \zeta_7y, \zeta_7^6t).\]
The collection of singular fibers consists of a type $II^*$ fiber at $t=\infty$ and seven type $II$ fibers over the zeros of $t^7+1$. The order seven automorphism $\sigma_7$ fixes the following: the smooth fiber $E$ of genus one over $t=0$, the central component $R$ of the $II^*$ fiber, and eight isolated points on the same fiber $II^*$. The automorphism $\sigma_{21}$ fixes $R$ as well and acts on $E$ as an automorphism of order three, fixing three points. The fixed locus of $\sigma_3$ consists of $R$, along with another rational curve in the fiber $II^*$, the zero section, the genus three 3-section $X\cap\{y=0\}$, and four isolated points on the fiber $II^*$.

Therefore, the invariants of $\Fix(\sigma_{21}^j),j=1,3,7$ are as in the fourth row of Table \ref{tab:21}.
\\
Another example of this type of automorphism is given by the following. Consider the equation $y^3=z^7+x^2w+xw^{11}$ in the weighted projective space $\mathbb P(10,7,3,1)_{x,y,z,w}$, and consider
the order 21 automorphism
\[\sigma_{21}:(x,y,z,w)\mapsto (x,\zeta_3y,z,\zeta_7w).\]
The curve $C_{y}\coloneqq \{y=0\}$ has genus three and is fixed by $\sigma_3$, and the curve $C_{w}\coloneqq\{w=0\}$ has genus one and is fixed by $\sigma_7$. The rational curve fixed by $\sigma_{21}$ is a rational component in the resolution of the $A_9$ singularity $(1:0:0:0)$.
\end{example}


\section{Order 28}\label{sec-order28}

We now prove a classification theorem for purely non-symplectic automorphisms of order $28$ recovering the results in \cite{Brandhorst}. Our result is the following:


\begin{proposition}

\label{thm28}

The fixed locus of a purely non-symplectic automorphism of order $28$ on a K3 surfaces is not empty and it consists of either:

\begin{enumerate}[(i)]
    \item The union of $N_{21}$ isolated points, where $N_{21}\in \{3,5\}$; or
    \item The disjoint union of a rational curve and $10$ isolated points.
\end{enumerate}
Moreover, all these possibilities occur. The examples of \cite[Table 3]{Brandhorst} fit the invariants of Table \ref{tab:28} below, which provides a more detailed description of the possible different fixed loci of $\sigma_{28}$ and its powers.

{\small
\begin{table}[H]
\hspace*{-1.5cm}
\begin{tabular} {c|c|c|c|c|c}
 $\Fix(\sigma_{28})$& $\Fix(\sigma_{14})$ & $\Fix(\sigma_{7})$ & $\Fix(\sigma_{4})$ & $\Fix(\sigma_{2})$&Example\\
\noalign{\hrule height 1.5pt}
 $\{p_1,\ldots,p_5\}$ &$\{p_1,\ldots,p_5,p_6,p_7\}$&$E\sqcup \{
p_1,p_2,p_3\}$ & $\{q_1,\ldots,q_7,p_1\}\sqcup R_1\sqcup R_2$ & $C_3\sqcup R_1\sqcup R_2$ &\ref{ex28-1}\\

  $\{p_1,p_2,p_3\}$ & $\{p_1,\ldots,p_7\}$  &  $E\sqcup \{p_1,q_1,q_2\}$ &  $C_3$ & $C_3\sqcup R_1\sqcup R_2$&\ref{ex28-2}\\ \noalign{\hrule height 1.5pt}

 $R\sqcup \{p_1,\ldots,p_{10}\}$ & $R\sqcup \{p_1,\ldots,p_{10},p_{11},p_{12}\}$ &  $E\sqcup R\sqcup\{p_1,\ldots,p_8\}$ & $\{p_1,\ldots,p_8\}\sqcup R\sqcup R_1$ & $C_6\sqcup R\sqcup R_1\sqcup\ldots\sqcup R_4$&\ref{ex28-3}\\
\end{tabular}\caption{Order 28}
\label{tab:28}
\end{table}
}




\end{proposition}


\begin{proof}

As explained in Section \ref{sec:background}, at any fixed point the automorphism $\sigma_{28}$ acts as multiplication by $A_{i,28}\coloneqq \begin{pmatrix} \zeta_{28}^{i+1} & 0 \\ 0 & \zeta_{28}^{28-i} \end{pmatrix}$ for $0\leq i \leq 13$, and we denote the number of points of type $A_{i,28}$ by $m_{i,28}$.

The holomorphic Lefschetz formula \eqref{eq-Lefhol} applied to $\sigma_{28}$ gives us the following linear system of equations
\[\begin{cases}
3 m_{6,28} &= 3 + 4 m_{1,28} - 5 m_{2,28} - 4 m_{4,28} + 8 m_{5,28} \\
 3 m_{7,28} &= 3 - 5 m_{1,28} + 4 m_{2,28} - 13 m_{4,28} + 17 m_{5,28} \\
 m_{8,28} &= 1 - 2 m_{1,28} + 2 m_{2,28} - 5 m_{4,28} + 6 m_{5,28} \\
 m_{9,28} &= 3 - 4 m_{1,28} + 4 m_{2,28} - 3 m_{3,28} - 3 m_{4,28} + 7 m_{5,28} \\
 2 m_{10,28} &= 2 - 3 m_{1,28} + 3 m_{2,28} - 2 m_{3,28} - 3 m_{4,28} + 6 m_{5,28} \\
 6 \alpha_{28} &= m_{1,28} + m_{2,28} - m_{4,28} + 2 m_{5,28} 
\end{cases}\]
where $\alpha_{28}\coloneqq \sum (1-g(C))$ and the sum runs over all curves $C$ which are fixed by $\sigma_{28}$.

Moreover, considering the automorphism $\sigma_{7}=\sigma_{28}^4$ which has order seven, we further know that 
\[\begin{cases}
m_{1,28}+m_{5,28}+m_{8,28}+m_{12,28} &\leq m_{1,7} \\
 m_{2,28}+m_{4,28}+m_{9,28}+m_{11,28} &\leq m_{2,7} \\
m_{3,28}+m_{10,28} &\leq m_{3,7} 
\end{cases}\]
Note that \begin{itemize}
    \item points of type $A_{13,28}$ lie on a curve fixed by $\sigma_{14}$ (but not by $\sigma_{28}$);
    \item points of type $A_{7,28}, A_{8,28}$ and $A_{13,28}$ lie on a curve fixed by $\sigma_7$ (but not fixed by $\sigma_{28}$);
    \item points of type $A_{j,28}, j=3,4,7,8,11,12$ lie on a curve fixed by $\sigma_4$ (but not fixed by $\sigma_{28}$).
\end{itemize}

Because of the observations listed above, letting $r\coloneqq m_{6,28}+m_{7,28}+m_{13,28}$ we obtain the following four possibilities for $(m_{1,28},\ldots,m_{13,28};\alpha_{28},r)$:
\begin{align*}
w_1=(0,0,0,0,0,0,2,2,1,0,0,0,0;0,2) && w_3=(0,0,0,1,0,0,2,0,0,0,0,0,0;0,2) \\
w_2=(3,2,1,0,1,2,0,2,1,0,0,0,0;1,2) && w_4=(3,2,1,1,1,2,0,0,0,0,0,0,0;1,2)
\end{align*}

We now consider the induced action of $\sigma_{28}$ on $H^2(X,\mathbb{R})$ and as in Section \ref{sec:background} we let
\[d_{i} \coloneqq \dim H^2(X,\mathbb{R})_{\zeta_{i}},\ i=28,14,7,4,2,1\] 


For each $i=2,4,7,14,28$ we let $\chi_i$ denote the Euler characteristic of the fixed locus of the power of $\sigma_{28}$ which has order $i$. 
Applying the topological Lefschetz formula \eqref{eq-Leftop} to $\sigma_{28},\sigma_{14},\sigma_7,\sigma_4$ and $\sigma_2$ we obtain:
\begin{equation}\label{top28}
\begin{cases}
\chi_{28} &= 2 + d_{14} - d_7 - d_2 + d_1 \\
\chi_{14} &= 2 + 2d_{28} - d_{14} - d_7 - 2d_4 + d_2 + d_1 \\
\chi_7 &= 2 - 2d_{28} - d_{14} - d_7 + 2d_4 + d_2 + d_1\\
\chi_4 &= 2 - 6d_{14} + 6d_7 - d_2 + d_1\\
\chi_2 &= 2 - 12d_{28} + 6d_{14} + 6d_7 - 2d_4 + d_2 + d_1
\end{cases}\end{equation}

Moreover, we know that 
\[
22 = \text{dim }H^2(X,\mathbb{R})= 12 d_{28} + 6d_{14} + 6d_7 + 2d_4 + d_2 + d_1 
\]

Using \eqref{top28} one gets the following possibilities, according to the four vectors $w_i$:
{\small \begin{center}
\begin{tabular} {c|c|c|c|c|c|c}
$w_i$ & $\chi_{28}$ & $(d_{28},\ldots,d_1)$& $\chi_{14}$ & $\chi_7$& $\chi_{4}$ &$\chi_2$     \\
\hline
$w_1$ & 5&(1,1,0,1,0,2) &3 & 3 & -2&-4  \\
    && (1,1,0,0,1,3) &7&3&-2 &0\\
\rowcolor{Lgray}    && (1,0,1,0,0,4) &7&3&12&0\\
    \hline
  
$w_2$ & 14 & - & - & - &-&-\\
\hline
 $w_3$ & 3 & (1,1,0,1,0,1) &3 &3 &-4&-4 \\
&  & (1,0,1,1,0,2) & 3& 3&10&-4 \\
\rowcolor{Lgray} &  & (1,1,0,0,2,2) & 7& 3 & -4 &0 \\
&  & (1,0,1,0,1,3) & 7& 3 & 10 &0 \\

\hline
\rowcolor{Lgray} $w_4$ &12& (1,0,0,0,0,10) &14&10&12&0\\

\end{tabular}
\end{center}}

Observe that vector $w_2$ does not give any admissible case and $\chi_{14}$ cannot be 3 by our classification of Section \ref{sec-order14}.
This implies either $\chi_{7}=10$, the vector of types of points is $w_4$ and $\sigma_{14}=\sigma_{28}^2$ is of type B3 of Table \ref{tab:fixed} or  $(\chi_{14},\chi_7, \chi_2)= (7,3,0)$. In the latter case,  $\sigma_{14}$ is of type A1(3,2) of Table \ref{tab:fixed} and $\Fix(\sigma_2)=C_3\,\cup \, R_1 \, \cup \, R_2$. 
Recalling that $\Fix(\sigma_4) \subseteq\Fix(\sigma_2)$ and that $\sigma_4$ acts with order 1 or 2 on $\Fix(\sigma_2)$, with Riemann-Hurwitz formula we can conclude that $\chi_4\in\{-4,0,4,8,12\}$.

This leaves only the cases highlighted in gray. 

We now study the action of $\sigma_4$ on $\Fix(\sigma_2)$. If 
$\chi_4=-4$, then $\Fix(\sigma_4)=C_3$ and $R_1$ and $R_2$ are exchanged by $\sigma_4$. 
If $\chi_4=12$, then $\sigma_4$ can only fix rational curves and \cite[Proposition 1]{ArtebaniSarti4} implies $\sigma_4$ fixes exactly two rational curves and 8 isolated points. 
\end{proof}

\begin{remark}\label{d28}
As in the order 21 case, note that the following relations hold:
\[\rk S(\sigma_7)=2d_4+d_2+d_1,\quad \rk S(\sigma_4)=6d_7+d_1,\quad \rk S(\sigma_2)=6(d_{14}+d_7)+d_2+d_1.
\]
\end{remark}

We now show the examples in \cite[Table 3]{Brandhorst} are consistent with the invariants listed on Table \ref{tab:28}.

\begin{example}\label{ex28-1} The elliptic K3 surface with Weierstrass equation 
\[y^2=x^3+(t^7+1)x\]
admits the following order 28 purely non-symplectic automorphism  
\[\sigma_{28}(x,y,t)=(x-(y/x)^2,i(y-(y/x)^3),\zeta_7t),\]
The elliptic fibration admits a smooth fiber over $t=0$, a fiber of type $II$ over $t=\infty$ and 7 fibers of type $II$ over the roots of $\Delta=4(t^7+1)^3$.
One can check the invariants of $\Fix(\sigma_{28}^j), j=1,2,4,14$ are as in the first row of Table \ref{tab:28}. In particular, the automorphism $\sigma_{14}=\sigma_{28}^2$ is of type $A1(3,2)$ in our classification of Section \ref{sec-order14}. Moreover, given that $\Fix(\sigma_2)=C_3\sqcup R_1\sqcup R_{2},$ we have that $\sigma_{4}$ does not exchange $R_1$ and $R_2$ and fixes the 8 tangential points of the fibers of type II lying on $C_{3}.$ Therefore, $\sigma_{28}$ fixes the same three points in the fiber over $t=\infty$ and two additional points in smooth fiber over $t=0.$   \end{example}

\begin{example} \label{ex28-2}
The elliptic K3 surface with Weierstrass equation 
\[y^2=x^3+(t^7+1)x,\ t\in\mathbb P^1,\]
admits the following order 28 purely non-symplectic automorphism 
\[\sigma_{28}(x,y,t)=(-x,iy,-\zeta_7t).\]

One can check that the invariants of $\Fix(\sigma_{28}^j), j=1,2,4,14$ are as in the second row of Table \ref{tab:28}. In particular, the automorphism $\sigma_{14}=\sigma_{28}^2$ is of type $A1(3,2)$ in our classification of Section \ref{sec-order14}.
Given that $\Fix(\sigma_2)=C_3\sqcup R_1\sqcup R_{2},$ the automorphism $\sigma_{4}$ exchanges $R_1$ and $R_2$ and fixes $C_{3}.$ As a consequence, $\sigma_{28}$ fixes the tangential point in the fiber of type II over $t=\infty$ and two additional points in the smooth fiber.

Another example of this type of automorphism is given by the following. Consider the K3 surface in $\mathbb P(7,3,2,2)$ which is the zero locus of the quasi-smooth polynomial $x^2+y^4z+z^7+w^7$. It admits the purely non-symplectic automorphism of order 28 
\[\sigma_{28}(x,y,z,w)=(x,iy,z,\zeta_7w).\]
Resolving the singularity of type $A_2$ at $(0:1:0:0)$ and the seven singularities of type $A_1$ at $(0:0:\zeta_{14}^i, 1), i=1,3,\ldots,13$ we see that the different fixed loci of the powers of $\sigma_{28}$ are as in the second row of Table \ref{tab:28}.
\end{example}

\begin{example} \label{ex28-3} The elliptic K3 surface with Weierstrass equation 
\[y^2=x^3+x+t^7,\ t\in\mathbb P^1,\]
admits the order 28 purely non-symplectic automorphism  
\[\sigma_{28}(x,y,t)=(-x,iy,-\zeta_7t).\]
 The elliptic fibration admits a smooth fiber over $t=0$, a fiber of type $II^*$ over $t=\infty$ and 14 nodal curves over the roots of $\Delta=4+27t^{14}$. The automorphism $\sigma_{14}=\sigma_{28}^2$ is of type $B3$ in our classification in Section \ref{sec-order14} and we can check the invariants of $\Fix(\sigma_{28}^j), j=1,4,14$ indeed agree with the third row of Table \ref{tab:28}. Moreover, since $Fix(\sigma_2)=C_6\sqcup R\sqcup R_1\sqcup \dots R_4,$ we have that $\sigma_4$ fixes two rational curves including $R,$ two points in $C_6$ and six additional points in the other rational curves. As a consequence, $\sigma_{28}$ fixes  $R$ and ten additional points, two of them on the smooth fiber over $t=0.$ 
\end{example}

\section{Order 42}\label{sec-order42}

In \cite{Brandhorst}, Brandhorst classifies purely non-symplectic automorphisms of order $42$ on K3 surfaces. Here, we provide a different and more geometric view of his result. We prove:

\begin{proposition}

\label{thm42}

The fixed locus of a purely non-symplectic automorphism of order $42$ on a K3 surfaces is not empty and it consists of either:

\begin{enumerate}[(i)]
    \item The union of $N_{21}$ isolated points, where $N_{21}\in \{5,6\}$; or
    \item The disjoint union of a rational curve and $9$ isolated points.
\end{enumerate}
Moreover, all these possibilities occur, and a more detailed description is given in Table \ref{tab:42} below.

{\small 
\begin{table}[H]
\begin{tabular} {c!{\vrule width 1.5pt} c|c|c|c|c}
Type $\sigma_{14}$ & $\Fix(\sigma_{42})$ & $\Fix(\sigma_7)$ & $\Fix(\sigma_{21})$ & $\Fix(\sigma_3)$ &Example  \\
\noalign{\hrule height 1.5pt}
  C1 & $\{p_1,\ldots,p_6\}$&$R\sqcup\{p_1,\ldots,p_8\}$ & $R\sqcup\{p_1,\ldots,p_8\}$ & $C_3\sqcup R\sqcup R'\sqcup\{p_1,p_2,p_3\}$ &\ref{ex42-1}\\ \hline
  
  C3 & $\{p_1,\ldots,p_5\}$& $R\sqcup\{p_1,\ldots,p_8\}$ & $\{p_1,\ldots,p_7\}$ & $C_3\sqcup R\sqcup \{p_1,p_2\}$ &\ref{ex42-2}\\ \hline
  
  B3 & $R\ \sqcup\{p_1,\ldots p_9\}$&$E\sqcup R\sqcup\{p_1,\ldots,p_8\}$&$R\ \sqcup\{p_1,\ldots p_{11}\}$ & $C_3\sqcup R\sqcup R'\sqcup R'' \sqcup\{p_1,\ldots p_4\}$&\ref{ex42-3}
\end{tabular}\caption{Order 42}
\label{tab:42}
\end{table}}

\end{proposition}

\begin{proof}
Let $\sigma_{42}$ be a purely non-symplectic automorphism of order 42. 
Thus its square is a purely non-symplectic automorphism of order 21 and we use the classification of Section \ref{sec-order21}. 

Observe that isolated fixed points for $\sigma_{42}$ of type $A_{20,42}$ lie on curves fixed by $\sigma_{21}$ and not fixed by $\sigma_{42}$. Thus, if $\sigma_{21}$ has invariants as in the first or fourth rows of Table \ref{tab:21}, it must be the case that $m_{20,42}$ is either 0 or 2, according to the fact the the rational curve $R\subset\Fix(\sigma_{21})$ is fixed by $\sigma_{42}$ or not.

We also have the following inequalities
\[m_{1,42}+m_{19,42}\leq m_{1,21},\quad 
m_{2,42}+m_{18,42}\leq m_{2,21},\quad 
m_{3,42}+m_{17,42}\leq m_{3,21},\quad 
m_{4,42}+m_{18,42}\leq m_{4,21},\]
\[m_{5,42}+m_{15,42}\leq m_{5,21},\quad 
m_{6,42}+m_{14,42}\leq m_{6,21},\quad 
m_{7,42}+m_{13,42}\leq m_{7,21},\quad m_{8,42}+m_{12,42}\leq m_{8,21},\]
\[m_{9,42}+m_{11,42}\leq m_{9,21},\quad m_{10,42}\leq m_{10,21}\]

According to this, we look for possible solutions $m=(m_{1,42},\ldots,m_{20,42};\alpha_{42})$ of the Lefschetz holomorphic formula \eqref{eq-Lefhol} applied to $\sigma_{42}$. Using MAGMA we get the following: 
\begin{itemize}
\item if $\sigma_{21}$ is as in the first row of Table \ref{tab:21}, there is no possible solution $m$ with $\alpha_{42}=1$. If $\alpha_{42}=0$ one gets the vector $m=(0,0,0,0,0,0,0,1,0,0,0,0,0,0,0,0,1,1,1,2;0)$. Thus $\Fix(\sigma_{42})$ consists of 6 isolated points, two of which are contained in the rational curve fixed by $\sigma_{21}$.

\item if $\sigma_{21}$ is as in the second row of Table \ref{tab:21}, then  $\alpha_{42}$ is necessarily 0 and $m_{20,42}=0$. There is one solution
$m=(0,0,0,0,0,0,0,0,0,1,1,1,1,1,0,0,0,0,0,0;0)$.
Thus $\Fix(\sigma_{42})$ consists of 5 isolated points.

\item if $\sigma_{21}$ is as in the third row of Table \ref{tab:21},  $\alpha_{42}$ is necessarily 0 and $m_{20,42}=0$. There is no solution in this case.

\item if $\sigma_{21}$ is as in the fourth of Table \ref{tab:21}, $\alpha_{42}$ can be 0 or 1. If $\alpha_{42}=0$ and $m_{20,42}=2$ there are no solutions.
If $\alpha_{42}=1$ and $m_{20,42}=0$ by MAGMA we get only one solution
$(m_{1,42},\ldots,m_{20,42};\alpha_{42})=(3,2,1,1,1,1,0,0,0,0,0,0,0,0,0,0,0,0,0,0;1)$. Thus $\Fix(\sigma_{42})$ consists of a rational curve and 9 isolated points.

\end{itemize}

Thus, there are three possibilities for  $\Fix(\sigma_{42})$.
As before, let $d_{i}\coloneqq \dim H^2(X,\R)_{\zeta_{i}}$ for $i=42,21,14,7,6,3,2,1$.
We have \[
22=12d_{42}+12d_{21}+6d_{14}+6d_7+2d_6+2d_3+d_2+d_1\]
By the topological Lefschetz formula \eqref{eq-Leftop} applied to the powers of $\sigma_{42}$ we get the following linear system of equations
\begin{equation}\label{top42}
\begin{cases}
\chi_{42}&=2-d_{42}+d_{21}+d_{14}-d_7+d_6-d_3-d_2+d_1\\
\chi_{21}&=2+d_{42}+d_{21}-d_{14}-d_7-d_6-d_3+d_2+d_1\\
\chi_{14}&=2+(2d_{42}+d_{14})-(2d_{21}+d_7)-(2d_6+d_2)+2d_3+d_1\\
\chi_7&=2-(2d_{42}+2d_{21}+d_{14}+d_7)+2d_6+2d_3+d_2+d_1\\
\chi_6&=2+(6d_{42}+d_6)-(6d_{21}+d_3)-(6d_{14}+d_2)+6d_7+d_1\\
\chi_3&=2-(6d_{42}+6d_{21}+d_6+d_3)+6d_{14}+6d_7+d_2+d_1\\
\chi_2&=2-(12d_{42}+6d_{14}+2d_6+d_2)+12d_{21}+6d_7+2d_3+d_1
\end{cases}
\end{equation}
 Considering the different possible solutions we can compute the values of the Euler characteristics of the fixed locus of $\sigma_{42}$ and its powers:
 
\begin{center}
\begin{tabular} {c|c|c|c|c|c|c|c|c|c}
Type $\sigma_{14}$ & $\chi_{42}$ & $\Fix(\sigma_{42})$ &  $(d_{42},\ldots,d_1)$ &$\chi_{21}$ & $\chi_{14}$ & $\chi_7$ & $\chi_6$ & $\chi_{3}$ & $\chi_2$ \\
\hline
  C1 &  6&6 pts& (1,0,0,0,1,0,2,6) &  10 & 6 & 10 & 13 & 3 & -8 \\
  C3 &  5 & 5 pts & (1,0,0,0,1,1,1,5) &  7& 8 & 10 & 12 & 0 & -6 \\
  B3 &  11& $R\ \sqcup$ 9 pts & (1,0,0,0,0,0,0,10) & 13& 14 &10 & 18 & 6 & 0
\end{tabular}
\end{center}\end{proof}

\begin{remark}\label{d42}
Observe the following relations hold:
\[\rk S(\sigma)=d_1\geq 1,\quad \rk S(\sigma_7)=2d_6+2d_3+d_2+d_1,\quad 
\rk S(\sigma_6)=6d_7+d_1,\]\[
\rk S(\sigma_3)=6d_{14}+6d_7+d_2+d_1,\quad 
\rk S(\sigma_2)=6(2d_{21}+d_7)+2d_3+d_1.
\]
\end{remark}

\begin{remark}
The complete description of $\Fix(\sigma_3)$ follows from Proposition \ref{thm21}.
\end{remark}

\begin{remark}
The possible values of $\chi_6$ obtained in the proof of Proposition \ref{thm42} and the classification in \cite{Dillies} allow us to also completely describe $\Fix(\sigma_6)$. The description is as follows:

If $\sigma_{42}$ is as in the first row of Table \ref{tab:42}, then we must have $m_{2,6}=10$ and $m_{1,6}=1$. Moreover, $\sigma_6$ fixes 1 rational curve. With our notations, there are 8 fixed points under $\sigma_6$ lying on $C_3$, $\sigma_6$ fixes $p_1$, it also fixes $R$ and it has 2 more fixed points lying on $R'$.

Now, if $\sigma_{42}$ is as in the second row of Table \ref{tab:42}, then
$m_{2,6}=8$ and $m_{1,6}=2$. Moreover, $\sigma_6$ fixes 1 rational curve. There are 8 points fixed under $\sigma_6$ lying on $C_3$, $\sigma_6$ fixes $p_1$ and $p_2$, and it also fixes $R$.

Finally, if $\sigma_{42}$ is as in the last row of Table \ref{tab:42}, then  $m_{2,6}=10$ and $m_{1,6}=4$. Moreover, $\sigma_6$ fixes 2 rational curves. There are 8 points fixed under $\sigma_6$ lying on $C_3$, $\sigma_6$ fixes $p_1,\ldots,p_4$, it also fixes $R$ and $R'$ and it has 2 more fixed points lying on $R''$.
\end{remark}

Observe that Proposition \ref{thm42} is compatible with \cite{Brandhorst}. In fact the examples in \cite[Table 3]{Brandhorst} agree with the invariants listed on Table \ref{tab:42}, as we describe below:

\begin{example} \label{ex42-1}The K3 surface is the same as in Example \ref{ex21-C1}, see \cite{Brandhorst}. On the same elliptic fibration $y^2=x^3+4t^4(t^7-1)$ the order 42 automorphism $\sigma_{42}$ is given by \[ \sigma_{42}: (x,y,t)\mapsto (\zeta_7^6\zeta_3 x, -\zeta_7^2y, \zeta_7t).\]
The automorphism $\sigma_{42}$ acts on the fiber of type $IV^*$ as a reflection, moving two legs and leaving the third invariant. Thus on the fiber $IV^*$ $\sigma_{42}$ fixes 4 isolated points. The 2 isolated points fixed by $\sigma_{21}$ on the cuspidal fiber over $t=\infty$ are fixed by $\sigma_{42}$ too.
In particular, the invariants of $\Fix(\sigma_{42}^j),j=1,2,3,6,7,14,21$ are as in the first row of Table \ref{tab:42}.
\end{example}

\begin{example} \label{ex42-2}The K3 surface is the same as in Example \ref{ex21-C2}, see \cite{Brandhorst}.
On the same elliptic fibration
$y^2=x^3+t^3(t^7+1)$ the order 42 automorphism $\sigma_{42}$ is given by \[ \sigma_{42}: (x,y,t)\mapsto (\zeta_7^3\zeta_3 x, -\zeta_7y, \zeta_7^3t).\]
One can check that the invariants of $\Fix(\sigma_{42}^j),j=1,2,3,6,7,14,21$ are as in the second row of Table \ref{tab:42}.
\end{example}

\begin{example} \label{ex42-3}The K3 surface is the same as in Example \ref{ex21-B4}, see \cite{Brandhorst}.
On the same elliptic fibration
$y^2=x^3+t^5(t^7+1)$ the order 42 automorphism $\sigma_{42}$ is given by
\[\sigma_{42}: (x,y,t)\mapsto (\zeta_{42}^2 x, \zeta_{42}^3y, \zeta_{42}^{18}t)\]
 On the fiber $II^*$ $\sigma_{42}$ fixes 8 isolated points and the central component $R$. It also fixes 1 point on the elliptic curve $E$ over $t=0$.
 Therefore, the invariants of $\Fix(\sigma_{42}^j),j=1,2,3,6,7,14,21$ are as in the third row of Table \ref{tab:42}.
 \end{example}

\section{Not purely non-symplectic automorphisms}
\label{sec-NP}

As we observed in Section \ref{sec:background}, a not purely non-symplectic automorphism $f$ is such that its action on the period $\omega_X$ is given by multiplication by a non-primitive $n$-th root of unity (different from 1). As a consequence, at least one power of $f$ is symplectic.

The following are well known results about symplectic automorphisms on K3 surfaces. First, by \cite{Nikulin}, a symplectic automorphism can only fix isolated points, and its order must be less than or equal to eight. Moreover, according to the possible orders:

\begin{lemma} \label{lemma-symp}(see \cite[Prop 1.1]{GS1},\cite[Prop. 5.1]{GS2}, \cite{Nikulin})
Given a symplectic automorphism $g$ on a K3 surface, the number $N$ of isolated fixed points and the rank of the invariant lattice $S(g)$ are shown in the following table.
\begin{center}
\begin{tabular}{c|c|c!{\vrule width 1.5pt}c|c|c}
     ord(g)& $N$ &$\rk S(g)$&ord(g)& $N$ &$\rk S(g)$ \\
     \noalign{\hrule height 1.5pt}
     2& 8& 14 &6& 2& 6\\
     3& 6& 10 &7& 3& 4\\
     4& 4& 8 &  8& 2& 2\\
     5& 4& 6\\
    
\end{tabular}
\end{center}
\end{lemma}


In this section we will  provide a complete classification of not purely non-symplectic automorphisms of orders $14,21,28$ and $42$ according to which powers of the automorphisms are assumed to be symplectic. 

\subsection{Order 14}\label{sec-NP14}

Let $\sigma_{14}$ be a non-symplectic automorphisms of order 14 such that either $\sigma_7=\sigma_{14}^2$ or $\sigma_2=\sigma_{14}^7$ are symplectic.
We will study the two cases separately.

\subsubsection{$\sigma_7$ symplectic}

When the square of $\sigma_{14}$ is symplectic we prove:

\begin{proposition}\label{NPorder14}
Let $\sigma_{14}$ be a non-symplectic automorphism of order 14 on a K3 surface $X$ such that $\sigma_7=\sigma_{14}^2$ is symplectic. Then $\Fix(\sigma_{14})$ consists of 3 isolated points and the possible values of $(d_{14},d_7,d_2,d_1)$ are  $(2, 1, 2, 2), (3, 0, 3, 1)$.
In the first case, $\Fix(\sigma_2)$ consists of a curve of genus 3, while in the second case, it consists of a curve of genus 10. Moreover, both possibilities occur.
\end{proposition}

\begin{proof}
By Lemma \ref{lemma-symp}, the fixed locus of $\sigma_7$ consists of 3 isolated points and since $\Fix(\sigma_{14}) \subset\Fix (\sigma_7)$, it follows that the number $N_{14}$ of isolated points fixed by $\sigma_{14}$ is at most 3. Now, by Lemma \ref{lemma-symp}, we also know the invariant lattice of $\sigma_7$ has rank 4. Therefore, by Remark \ref{d14}, we further know $d_2+d_1=4$ and $d_{14}+d_7=3$. Moreover, by the topological Lefschetz formula \eqref{eq-Leftop} (applied to $\sigma_{14}$) we have 
\[
\chi_{14}=N_{14}=2+d_{14}-d_7-d_2+d_1.
\]
Further observing that we must have $d_2>0$ and $d_1>0$, these give the following list of possibilities for $(d_{14},d_7,d_2,d_1; N_{14})$: 
\begin{align*}
(0,3,1,3;1), (1, 2, 1, 3; 3), (1, 2, 2, 2; 1), (2, 1, 2, 2; 3), (2, 1, 3, 1; 1), (3, 0, 3, 1; 3).
\end{align*}
 As in Section \ref{sec-poss}, using $\eqref{Lef-top}$ and \cite{Nikulin-inv}, we can compute $\chi_2$ and the possible invariants $(g_2,k_2)$ of the fixed locus of $\sigma_2$. These are listed in Table \ref{tab:sigma^2symp} below. In particular, we observe that if $(d_{14},d_7,d_2,d_1)=(0,3,1,3)$, then we would have $\chi_2=22$, which is impossible by \cite{Nikulin-inv}. 
 
 \begin{table}[H]
    \centering
    \begin{tabular}{c|cccc|c|c}
         $N_{14}$&$d_{14}$&$d_7$&$d_2$&$d_1$&$\chi_2$&$(g_2,k_2)$\\
         \hline
         1&0&3&1&3&22&-\\
         1&1&2&2&2&8&$(3,6), (2,5), (1,4), (0,3)$\\
         1&2&1&3&1&-6&$(6,2), (5,1), (4,0)$\\
         3&1&2&1&3& 10 &$(2,6), (1,5), (0,4)$\\
         3&2&1&2&2& -4&$(6,3),(5,2),(4,1), (3,0)$ \\
         3&3&0&3&1&-18&$(10,0)$
    \end{tabular}
\caption{}
\label{tab:sigma^2symp}

\end{table}
 
 With computations similar to the ones of Section \ref{sec-exclude}, we can actually eliminate many of the other possibilities. In fact we see we must have that $\Fix(\sigma_{14})=\Fix(\sigma_7)$ consists of 3 isolated points and $\Fix(\sigma_2)$ consists of either a curve of genus 3 or a curve of genus 10.
The existence of both cases is shown in the following examples.

\end{proof}

\begin{example}
Let $f(x_0, x_1, x_2):=x^3_0x_1 + x^3_1x_2 + x^3_2x_0$ and consider the K3 surface
\[X_f:=\{(x_0 : x_1 : x_2 : x_3) : x^4_3= f(x_0, x_1, x_2)\}\subset \mathbb P^3.\]
This surface carries the order 14 automorphism
$\sigma_{14} : (x_0 : x_1 : x_2 : x_3) \mapsto(\zeta^4_7x_0 : \zeta^2_7x_1 : \zeta_7x_2 : -x_3)$. 
We have $\Fix (\sigma_{14}) = \{ (1 : 0 : 0 : 0), (0 : 1 : 0 : 0),  (0 : 0 : 1 : 0)\}$
and $\Fix (\sigma^7)$ is given by the curve $\{x_3 = 0\}$, which has genus three. Note that $\sigma^2$ is symplectic.
\end{example}

\begin{example}
Let $X$ be the surface in $\mathbb P(3,1,1,1)_{x,y,z,w}$ given as the zero locus of $x^2+y^5z+z^5w+w^5y$. $X$ admits the action of the order 14 automorphism \[\sigma_{14}:(x,y,z,w)\mapsto (-x, \zeta_7^5y, \zeta_7^3z, \zeta_7^6w)\] 
whose square is symplectic and fixes the three points $\{(0:1:0:0),(0:0:1:0),(0:0:0:1)\}$. 
The fixed locus of $\sigma_2$ is the genus 10 curve $\{x=0\}\cap X$.
\end{example}

\subsubsection{$\sigma_2$ symplectic}

We now consider what happens when the involution $\sigma_{14}^7$ is symplectic and $\sigma_{14}^2$ is non-symplectic. 
About the fixed loci of $\sigma_{14}$ and its powers, we can prove the following:

\begin{proposition}\label{prop_NP14-2}
Let $\sigma_{14}$ be a non-symplectic automorphism of order 14 on a K3 surface $X$ and assume the involution $\sigma_2=\sigma_{14}^7$ is symplectic. Then Fix($\sigma_{14}$) consists of $N_{14}\leq 8$ isolated points and the possible values of $N_{14}$ and $(d_{14},d_7,d_2,d_1)$ are given in Table \ref{tab:sigma^7symp} below, together with $\chi_7$ in each case.

\centering
    \begin{table}[H]
    \begin{tabular}{c|cccc|c}
         $N_{14}$&$d_{14}$&$d_7$&$d_2$&$d_1$&$\chi_7$\\
         \hline
         1&0&1&8&8&17\\
         8&1&1&2&8&10\\
         
         1&1&2&2&2&3\\
    \end{tabular}\caption{}
    
    \label{tab:sigma^7symp}
    \end{table}

\end{proposition}

\begin{proof}
By \cite{Nikulin}, the fixed locus of the symplectic involution $\sigma_2$ consists of 8 isolated points. Since $\Fix(\sigma_{14}) \subseteq \Fix (\sigma_2)$, it follows that $N_{14}\leq 8$. The invariant lattice of $\sigma_2$ has rank 14 by Lemma \ref{lemma-symp}, thus $6d_7+d_1=14$ and $6d_{14}+d_2=8$ by Remark \ref{d14}. Moreover, by the topological Lefschetz formula (applied to $\sigma_{14}$) we have 
\[
\chi(\Fix(\sigma_{14}))=N_{14}=2+d_{14}-d_{7}-d_2+d_1.
\]
Further observing that we must have $d_7>0$ and $d_1>0$, this gives the above list of possibilities, that is, if the involution $\sigma_2$ is symplectic one has 3 possibilities for $(d_{14},d_7,d_2,d_1)$, namely 
$
(0, 1, 8, 8), (1, 1, 2, 8), (1, 2, 2, 2)
$.\end{proof}

According to \cite[Table 3]{brandhorst2021} there are four different deformation families of K3 surfaces with these automorphisms.
Assuming general conditions, i.e. that the Picard lattice of the surface coincides with $S(\sigma_7)$, we prove the following Lemma which complements Proposition \ref{prop_NP14-2} (and Table \ref{tab:sigma^7symp}).

\begin{lemma}
Let $\sigma_{14}$ be a non-symplectic automorphism of order 14 on a K3 surface and assume the involution $\sigma_2=\sigma_{14}^7$ is symplectic. Under the assumption that the Picard lattice agrees with $S(\sigma_7)$ we have that the order seven automorphism $\sigma_7=\sigma_{14}^2$ cannot be of type $\dagger$ (here we are referring to the notation in Table \ref{tab:7}). In particular, $\sigma_7$ must fix a curve.
\label{caseX}
\end{lemma}

\begin{proof}
Since the K3 surface admits a symplectic involution, the transcendental lattice $T_X$ must be primitively embedded in $E_8(2) \oplus U \oplus U \oplus U$ \cite{GS}. By assumption, $T_X = T(\sigma_7)$, and we see that $\sigma_7$ cannot be of type $\dagger$, since in that case $T(\sigma_7)=U(7)\oplus U \oplus E_8 \oplus A_6$.
\end{proof}

As a consequence of the Lemma, the four deformations families given in \cite[Table 3]{brandhorst2021} correspond to families of type A, B, C and D. One can check Table \ref{tab:sigma^7symp} to know $N_{14}$ in each case.

We now exhibit two examples of possibilities in Table \ref{tab:sigma^7symp}: one belonging to case C, corresponding to the second line of Table \ref{tab:sigma^7symp}, and one belonging to case A and  corresponding to the third line of Table \ref{tab:sigma^7symp}. Observe that in both cases, the K3 surfaces appear also in the classification of Section \ref{sec-order14}, showing that these surfaces admit both a symplectic and non-symplectic involution.

\begin{example} {\bf{(Case C)}} Let us consider the  elliptic K3 surface $X$ with Weierstrass equation given by $y^2=x^3+t^2x+t^{10},\ t\in \mathbb P^1$. The automorphism $\sigma_{14}\colon(x,y,z,t,s)\mapsto (\zeta_{7}x,-\zeta^5_{7}y,-\zeta_{7}t)$ is a non-symplectic automorphism of order 14 and $\sigma_2$ is symplectic. 
Note that $\Fix(\sigma_{14})$ consists of 8 points.
\end{example}

\begin{example} {\bf{(Case A)}} Let $X$ be the elliptic K3 surface given by an equation of the form
$
y^2=x(x^2+t^7+1), t\in\mathbb P^1$. Then $X$ admits the order 14 purely non-symplectic automorphism
$
\sigma: (x,y,t)\mapsto(x,-y,\zeta_7t)$ described in Example \ref{example_A1_(3,2)}, which corresponds to case A1 with $(g_2,k_2)=(3,2)$ in our classification of Section \ref{sec-order14}.
Composing $\sigma^2=\sigma_7: (x,y,t)\mapsto (x,y,\zeta_7t)$ and the translation 
\[
\tau: (x,y,t)\mapsto ((y/x)^2-x,(y/x)^3-y,t)
\]
by the 2-torsion section produces an automorphism of order 14, say $\varphi$, which is not purely non-symplectic. 
By construction, $\varphi^7=\tau$ and $\tau$ is symplectic. The invariants of $\varphi$ are as in the third row of Table \ref{tab:sigma^7symp}.
\end{example}

\subsection{Order 21}
\label{sec-NP21}

Let $\sigma_{21}$ be a non-symplectic automorphism of order 21 such that either $\sigma_7=\sigma_{21}^3$ or $\sigma_3=\sigma_{31}^7$ are symplectic.
Again, we will study the two cases separately.

\subsubsection{$\sigma_7$ symplectic}
\begin{proposition}
If $\sigma_{21}$ is a non-symplectic automorphism of order 21 on a K3 surface $X$, then $\sigma_7$ cannot be symplectic. 
\label{order7NP21}
\end{proposition}

\begin{proof}
By contradiction, assume $\sigma_7$ is symplectic. Then, by Nikulin, the fixed locus of $\sigma_7$ consists of 3 isolated points. Since $\Fix(\sigma_{21}) \subseteq\Fix (\sigma_7)$, $\sigma_{21}$ is acting with order three on $\Fix (\sigma_7)$. So $\Fix(\sigma_{21})$ consists of $N_{21}=0$ or $N_{21}=3$ isolated fixed points.

Now, because the invariant lattice of $\sigma_7$ has rank 4 (Lemma \ref{lemma-symp}), we also know $d_3+d_1=4$ and $d_{21}+d_7=3$. Moreover, by the topological Lefschetz formula \eqref{eq-Leftop} (applied to $\sigma_{21}, \sigma_7$ and $\sigma_3$) we have 
\[\begin{cases}
\chi_{21}&=N_{21}=2+d_{21}-d_7-d_3+d_1\\
\chi_7&= 2-(2d_{21}+d_7)+2d_3+d_1 = 3\\
\chi_3&= 2-6(d_{21}-d_7)-d_3+d_1
\end{cases}\]

By Remark \ref{d21} and further observing that $d_3,d_1>0$ these give the  possibilities for $(d_{21},d_7,d_3,d_1)$ and $\chi_3$ shown in Table \ref{tab21_not_purely}, but since $\chi_3=N_3+2(1-g_3+k_3)=N_3+2(N_3-3)=3N_3-6$ we can eliminate all cases.

\centering
    \begin{table}[H]
    \begin{tabular}{c|c|c}
         $N_{21}$ & $(d_{21},d_7,d_3,d_1)$ & $\chi_3$ \\
         \hline
          3 & (2,1,2,2) & -4 \\
3 & (1,2,1,3) & 10 \\
3 & (3,0,3,1) & -18 \\
    \end{tabular}
    
\caption{}
    \label{tab21_not_purely}
\end{table}

\end{proof}

\subsubsection{$\sigma_3$ symplectic}

Similarly, we can prove:

\begin{proposition}\label{prop21NP}
Let $\sigma_{21}$ be a non-symplectic automorphism of order 21 on a K3 surface $X$ and assume $\sigma_3$ is symplectic. Then $\Fix(\sigma_{21})$ consists of  exactly $N_{21}=6$ isolated points and the only possible values for $(d_{21},d_7,d_3,d_1)$ are $(1, 1, 0, 4)$. 
Moreover, $\Fix(\sigma_7)$ is as in case A of Table \ref{tab:7} and such an automorphism exists (see Example \ref{egNP21}).
\end{proposition}

\begin{proof}
By Lemma \ref{lemma-symp} , the fixed locus of $\sigma_3$ consists of 6 isolated points. Since $\Fix(\sigma_{21}) \subseteq\Fix (\sigma_3)$, it follows that $\Fix(\sigma_{21})$ consists of $N_{21}\leq 6$ isolated fixed points. The invariant lattice of $\sigma_3$ has rank 10 by Lemma \ref{lemma-symp}, so we also know $6d_7+d_1=10$ and $6d_{21}+d_3=6$. Further observing that $d_7,d_1>0$, the topological Lefschetz formula \eqref{eq-Leftop} (applied to $\sigma_{21}, \sigma_7$ and $\sigma_3$) gives $(d_{21},d_7,d_3,d_1)=(1,1,0,4)$, $N_{21}=6$ and $\chi_7=3$.

Note that since $\Fix(\sigma_{21})= \{6\ pts\} \subseteq \Fix (\sigma_7)$, the above implies $\sigma_7$ is of type $A$. That is, $\Fix (\sigma_7)=E\,\cup \text{3 pts}$ and we must have three fixed points under $\sigma_{21}$ lying on $E$.
\end{proof}

\begin{example}
In $\mathbb P(3,2,1,1)$ we consider the surface 
$$x^2w+xy^2+yw^5+z^7=0$$
with the order 21 automorphisms \[\sigma_{21}:(x,y,z,w)\mapsto (\zeta_3x, \zeta_3y, \zeta_7z, \zeta_3w)\]

The order 7 automorphism $\sigma_7$ is non-symplectic and  fixes the genus 1 curve $\{z=0\}$  and 3 more points on the resolutions of the singularities  (1:0:0:0) and  (0:1:0:0), of type $A_2$ and $A_1$ respectively.
The automorphism $\sigma^7=\sigma_3$ is symplectic. 
\label{egNP21}
\end{example}

\subsection{Order 28}
\label{sec-NP28}
Let $\sigma_{28}$ be a non-symplectic automorphisms of order 28. We will prove in what follows that no power of $\sigma_{28}$ can be symplectic. 

\begin{proposition}
If $\sigma_{28}$ is a non-symplectic automorphism of order 28, then $\sigma_{28}$ is purely non-symplectic. In other words, no power of $\sigma_{28}$ is symplectic. 
\end{proposition}

\begin{proof}
We will assume some power of $\sigma_{28}$ is symplectic. Since there are no symplectic
automorphisms of (finite) order bigger than 8 by \cite{Nikulin}, we have to consider three posibilities:
\begin{enumerate}[{Case} I]
    \item $\sigma_7=\sigma_{28}^4$ is a symplectic automorphism of order 7; or
    \item $\sigma_4=\sigma_{28}^7$ is a symplectic automorphism of order 4; or
    \item $\sigma_2=\sigma_{28}^{14}$ is a symplectic involution. 
\end{enumerate}

Observe that in this last case, $\sigma_4=\sigma_{28}^7$ is also symplectic.

We refer to Section \ref{sec-order28} for the definition of $(d_{28},d_{14},d_7,d_4,d_2,d_1)$ and recall the relations given in Remark \ref{d28}:
\[\rk S(\sigma_7)=2d_4+d_2+d_1,\quad \rk S(\sigma_4)=6d_7+d_1,\quad \rk S(\sigma_2)=6(d_{14}+d_7)+d_2+d_1.
\]
We now study the three cases separately.

\begin{enumerate}
\item[{\bf{Case I}}] {If $\sigma_7$ is a symplectic automorphism of order 7, the action of $\sigma_7^*$ on the period $\omega_X$ of $X$ is 
trivial. Therefore $\sigma_{28}^*\omega_X=\zeta_4\omega_X$, which implies that $d_4=\dim H^2(X,\R)_{\zeta_4}\geq 1$.
Then $(\sigma_{14})^*\omega_X=\pm\omega_X$ and since there are no symplectic
automorphisms of finite order bigger than 8 by \cite{Nikulin}, then $(\sigma_{14})^*\omega_X=-\omega_X$.
By Lemma \ref{lemma-symp}, the fixed locus of a symplectic automorphism of order 7 consists of 3 isolated points. Since $\Fix(\sigma_{28})\subseteq\Fix(\sigma_7)$, then $\sigma_{28}$ only fixes isolated points and their number is $N_{28}\leq 3$. Moreover, $\sigma_{28}$ acts with order 1,2 or 4 on $\Fix(\sigma_7)$, hence $N_{28}=1$ or 3. By Lemma \ref{lemma-symp}, $\rk S(\sigma_7)=4$; it follows by the above formulas and \eqref{top28} that 
\[\begin{cases}
    2d_4+d_2+d_1=4&\\
    6(2d_{28}+d_{14}+d_7)=18&\\
    \chi_{28}=N_{28}=2+d_{14}-d_7-d_2+d_1\\
\end{cases}\]

This gives the following list of possibilities for $(d_{28},d_{14},d_7,d_4,d_2,d_1)$:

\begin{align}
  ( 1, 0, 1, 1, 1, 1 ),
    ( 0, 1, 2, 1, 1, 1 ),
    ( 0, 0, 3, 1, 0, 2 ),
    ( 1, 1, 0, 1, 1, 1 ),\nonumber \\
    ( 0, 2, 1, 1, 1, 1 ),
    ( 1, 0, 1, 1, 0, 2 ),
    ( 0, 1, 2, 1, 0, 2 ) \label{vec}.\end{align}
    
Moreover observe that $\sigma_{14}$ is non-symplectic, $\sigma_{7}$ is symplectic and $\sigma_7=(\sigma_{14})^2$. Thus we can use the classification of not purely non-symplectic automorphisms of order 14 given in Proposition \ref{NPorder14}. 
In this case, the possible values of $(a',b',c',d')=(d_{14},d_7,d_2,d_1)$ are 
$(2,1,2,2),\ (3,0,3,1)$
and the relations with $(d_{28},d_{14},d_7,d_4,d_2,d_1)$ are
\[d_{28}=a',\ d_{14}+d_7=b',\ c'=d_4,\ d'=d_2+d_1. \]
The vectors in \eqref{vec} do not satisfy the above conditions, thus it is not possible for $\sigma_7$ to be symplectic.}

\item[{\bf{Case II}}] {Assume $\sigma_{4}$ is symplectic. 
Since $\sigma_4^*\omega_X=\omega_X$, then $\sigma_{28}^*\omega_X=\zeta_7\omega_X$, which implies $d_7\geq 1$.
By Lemma \ref{lemma-symp}, the fixed locus of $\sigma_4$ consists of 4 isolated points. Since $\Fix(\sigma_{28}) \subseteq \Fix (\sigma_4)$, it follows
that $\sigma_{28}$ only fixes $N_{28}$ isolated points with $N_{28}\leq 4$. Moreover, $\sigma_{28}$ acts with order  1 or 7 on $\Fix (\sigma_4)$, hence $N_{28} = 4$.
By Lemma \ref{lemma-symp}, $\rk S(\sigma_4)=8$, thus it follows from the above formulas and \eqref{top28} that 
\[\begin{cases}
    6d_7+d_1=8&\\
    12d_{28}+6d_{14}+2d_4+d_2=14&\\
    \chi_{28}=N_{28}=2+d_{14}-d_7-d_2+d_1=4\\
\end{cases}\]
The only solution is $(d_{28},d_{14},d_7,d_4,d_2,d_1)=(0,1,1,4,0,2)$.
Observe that in this case $\sigma_{28}^2=\sigma_{14}$ is non-symplectic. By \eqref{top28}, we can compute $\chi_{14}=-6$  and this is impossible since by \cite{ArtebaniSartiTaki}, the Euler characteristic of the fixed locus of a non-symplectic automorphism of order 14 is bigger than 0. Thus there are no possibilities for $(d_{28},d_{14},d_7,d_4,d_2,d_1)$ and hence $\sigma_4$ can't be symplectic.}

\item[{\bf{Case III}}] {We now show 
that there is no K3 surface with a non-symplectic automorphism $\sigma_{28}$ such that $\sigma_4=\sigma_{28}^7$ is non-symplectic and $\sigma_{2}=\sigma_{28}^{14}$ is symplectic.
Assume the involution $\sigma_2$ is symplectic and $\sigma_4$ is non-symplectic. Thus 
\[\sigma_2^*\omega_X=\omega_X,\quad \sigma_{28}^*\omega_X=\zeta_{14}^i\omega_X,\quad \sigma_4^*\omega_X\neq\omega_X.\]
Thus we are interested in odd $i$'s, such that $\sigma_4^*\omega_X=-\omega_X$. In particular, $d_{14}\geq 1$.
By Lemma \ref{lemma-symp}, the fixed locus of $\sigma_2$ consists of 8 isolated points and since $\Fix(\sigma_{28}) \subseteq \Fix (\sigma_2)$,
it follows that $\sigma_{28}$ only fixes $N_{28}$ isolated points and $N_{28}\leq 8$. Moreover, $\sigma_{28}$ acts on $\Fix(\sigma_2)$ with order 1, 2, 7 or 14; it follows that
either $N_{28}$ is even or $N_{28}=1$. 
By Lemma \ref{lemma-symp}, $\rk S(\sigma_2)=14$, thus it follows from the above formulas and \eqref{top28} that 
\[\begin{cases}
    6d_{14}+6d_7+d_2+d_1=14&\\
    12d_{28}+2d_4=8&\\
    \chi_{28}=N_{28}=2+d_{14}-d_{7}-d_2+d_1=4\\
\end{cases}\]
This gives the following list of possibilities for $(d_{28},d_{14},d_7,d_4,d_2,d_1)$:
\begin{equation}\label{vec1}( 0, 2, 0, 4, 1, 1 ),
    ( 0, 1, 1, 4, 1, 1 ),
    ( 0, 2, 0, 4, 0, 2 ),
    ( 0, 1, 1, 4, 0, 2 ),
    ( 0, 1, 0, 4, 5, 3 ).
\end{equation}
Moreover, observe that $\sigma_{14}$ is non-symplectic, $\sigma_{2}$ is symplectic and $\sigma_2=(\sigma_{14})^7$. Thus we can use the classification of not purely non-symplectic automorphisms of order 14 given in Proposition \ref{NPorder14}. 
In Proposition \ref{prop_NP14-2} we found three possible vectors $(a',b',c',d')=(d_{14},d_7,d_2,d_1)$:
\[(0, 1, 8, 8),\ (1, 1 , 2,  8),\ (1, 2, 2, 2)\]
and the relations with $(d_{28},d_{14},d_7,d_4,d_2,d_1)$ are, as before,
\[d_{28}=a', d_{14}+d_7=b', c'=d_4, d'=d_2+d_1. \]
The vectors in \eqref{vec1} do not satisfy the above conditions, thus it is not possible for $\sigma_2$ to be symplectic.

Therefore, we proved that a non-symplectic automorphism of order 28 is necessarily purely non-symplectic.}
\end{enumerate}
\end{proof}

\subsection{Order 42}

Let $\sigma_{42}$ be a non-symplectic automorphism of order 42. We will prove in what follows that $\sigma_{42}^{14}=\sigma_3$ can be symplectic, but any other power $\sigma_{42}^k$, where $k=6,7$ or $21$ must be non-symplectic. Note that there are no symplectic automorphisms of (finite) order bigger than 8 by \cite{Nikulin}.

We first prove the following:
\begin{proposition}\label{Prop42NP1}
Let $\sigma_{42}$ be a non-symplectic automorphism of order 42 on a K3 surface $X$,  and assume $\sigma_3=\sigma_{42}^{14}$ is symplectic. Then Fix$(\sigma_{42})$ consists of 2 or 4 isolated points. In the first case $\Fix(\sigma_{14})$ is as in case A2(9,0) of Table \ref{tab:fixed}; and, in the second, it is as in case A1(9,1). Moreover, both cases exist (see Example \ref{egNP42} and Example \ref{egNP42.2}).
\label{NP42}
\end{proposition}

\begin{proof}
Let $\sigma=\sigma_{42}$. If $\sigma^{14}=\sigma_3$ is symplectic. Then $\sigma^2=\sigma_{21}$ is a non-symplectic automorphism of order $21$ such that $\sigma_{21}^7$ is symplectic. Therefore, we can apply Proposition \ref{prop21NP} to conclude that $\sigma_{21}$ fixes exactly 6 points and $\sigma_7=\sigma^6$ is of type A of Table \ref{tab:7}. 

Now, if we let $a=d_{42}+d_{21}, b=d_{14}+d_7,c=d_6+d_3, d=d_2+d_1$, and $d_{i}\doteq \dim H^2(X,\mathbb{R})_{\zeta_{i}}$ for $i=42,21,14,7,6,3,2,1$, then Proposition \ref{prop21NP} also gives us $(a,b,c,d)=(1,1,0,4)$.

Combining the above with the Topological Lefschetz formula \eqref{eq-Leftop} applied to the powers of $\sigma$ as in (\ref{top42}) (and imposing the relations in Remark \ref{d42}) gives the following list of possible values for $(d_{42},d_{21},d_{14},d_{7},d_6,d_3,d_2,d_1)$:

\begin{align*}
(1,0,1,0,0,0,2,2),(1,0,1,0,0,0,1,3)
\end{align*}

Note that $\sigma^*\omega=\zeta_{14}^i \omega$ for some $1\leq i \leq 13$, and if $i$ is even (resp. $i=7$), then $\sigma^7$ (resp. $\sigma^2$) is symplectic, but the latter is impossible by Proposition \ref{prop21NP} and Lemma \ref{lemma-symp}. Thus, $i$ is odd ($\neq 7$) and $d_{14}\geq 1$. In other words, $\sigma^3=\sigma_{14}$ is purely non-symplectic of order $14$. In particular, $\chi_{2}\in \{0,-16,-14\}$, by Proposition \ref{thm14}.

In addition, note also that using (\ref{top42}), the first vector gives us $(\chi_{42},\chi_{21},\chi_{14},\chi_7,\chi_6,\chi_3,\chi_2)=(2,6,5,3,2,6,-16)$, while the second gives $(\chi_{42},\chi_{21},\chi_{14},\chi_7,\chi_6,\chi_3,\chi_2)=(4,6,7,3,4,6,-14)$.

Therefore, $\Fix(\sigma)$ consists of 2 or 4 isolated points. In the first case, $\sigma_{14}$ is of type A(9,0) of Table \ref{tab:fixed} and, in the second, $\sigma_{14}$ is of type A(9,1).
\end{proof}

\begin{example}
In $\mathbb P(4,2,1,1)_{x,y,t,s}$ we consider the K3 surface given by 
$$x^2+s^7t+y^4+yt^6$$
with the order 42 automorphism  \[\sigma_{42}:(x,y,z,w)\mapsto (-x, y, \zeta_7\zeta_3t, \zeta_3^2s).\]

The order 14 automorphism $\sigma_{14}$ is purely non-symplectic of type $A2(9,0),$ i.e.\ $\sigma_{7}$ fixes the genus 1 curve $\{s=0\}$  and 3 points, two of them on the resolutions of the singularities of $(1:-1:0:0)$ and $(1:1:0:0)$ of type $A_1.$ Moreover, we have that $\sigma_{2}$ fixes a genus 9 curve $\{x=0\}.$
The automorphism $\sigma^{14}=\sigma_3$ is symplectic. 

\label{egNP42}
\end{example}

\begin{example}
In $\mathbb P(5,3,1,1)_{x,y,t,s}$ we consider the K3 surface given by
$$x^2+s^7y+y^3t+t^{10}$$
with the order 42 automorphism  \[\sigma_{42}:(x,y,z,w)\mapsto (-x,\zeta_3 y, t, \zeta_7\zeta_3^2s).\]

The order 14 automorphism $\sigma_{14}$ is purely non-symplectic of type $A1(9,1),$ i.e.\ $\sigma_{7}$ fixes the genus 1 curve $\{s=0\}$  and 3 points, two of them on the resolutions of the singularities of $(0:1:0:0)$ of type $A_2.$ Moreover, we have that $\sigma_{2}$ fixes a genus 9 curve $\{x=0\}$ and a rational curve. 
The automorphism $\sigma^{14}=\sigma_3$ is symplectic. 

\label{egNP42.2}
\end{example}

In contrast, we further prove:
\begin{proposition} \label{prop42NP}
Let $\sigma_{42}$ be a non-symplectic automorphism of order 42 on a K3 surface $X$. Then $\sigma_{42}^k$ is non-symplectic for $k=6,7,21$.
\end{proposition}

\begin{proof}
By contradiction, assume $\sigma=\sigma_{42}$ is a non-symplectic automorphism of order 42 such that one of the powers $\sigma_{42}^k$ is symplectic with $k=6,7$ or $21$. That is, assume there exists a $k\in \{6,7,21\}$ such that the action of $\sigma$ on $\omega_X$ is given by multiplication by $\zeta_k^i$ for some $1\leq i < k$.

\begin{enumerate}
    \item[{\boldmath{$k=6$}}] {If $\sigma^{6}=\sigma_7$ is symplectic, then $\sigma^2=\sigma_{21}$ is a non-symplectic automorphism of order $21$ such that $\sigma_{21}^3$ is symplectic. But this contradicts Proposition \ref{order7NP21}.
}
\item[{\boldmath{$k=7$}}]{If $\sigma^7$ is symplectic, then $\sigma^{14}$ is also symplectic. And since Fix$(\sigma)\subset$ Fix$(\sigma^7)\subset$ Fix$(\sigma^{14})$, the proof of Proposition \ref{NP42} implies we must have 
\[
(d_{42},d_{21},d_{14},d_{7},d_6,d_3,d_2,d_1)=
(1,0,1,0,0,0,2,2)
\]
where $d_{i}\doteq \dim H^2(X,\mathbb{R})_{\zeta_{i}}$ for $i=42,21,14,7,6,3,2,1$. But for this vector we do not have $d_7\geq 1$. Therefore, $\sigma^{7}$ cannot be symplectic.}
\item[{\boldmath{$k=21$}}]{Finally, assume $\sigma^{21}=\sigma_2$ is symplectic. Then $
\sigma^*\omega_X=\zeta_{21}^i \omega_X$ for some $1\leq i \leq 20$. If $i=7$, then $\sigma^3$ is symplectic, which is impossible by \cite{Nikulin}. If $i=3$, then $\sigma^7$ would be symplectic, which we showed is not possible (case $k=7$). Therefore, $d_{21}\geq 1$.

Now, observe $\sigma_2=\sigma_{14}^7$ so that we can use our classification results from Section \ref{sec-NP14} to conclude $\sigma_2$ cannot be symplectic. 

More precisely, applying Proposition \ref{prop_NP14-2}, and letting $a=d_{42}+d_{14}, b=d_{21}+d_7, c=d_2+d_6, d=d_3+d_1$, we find that the possible values of $(a,b,c,d)$ are $(0,1,8,8), (1,1,2,8)$ or $(1,2,2,2)$. But for these three vectors, if we use (\ref{top42}) and Remark \ref{d42} together with the fact that $d_{21},d_1\geq 1, \chi_2=8,\chi_{42}\leq 8$ and 
\[
22=12(d_{42}+d_{21})+6(d_{14}+d_7)+2(d_6+d_3)+d_2+d_1  
\]
then we find no possible solutions for  $(d_{42},d_{21},d_{14},d_{7},d_6,d_3,d_2,d_1)$.}
\end{enumerate}
\end{proof}

\section{The N\'eron--Severi lattice}\label{sec-NS}

We conclude with a description of the N\'eron--Severi lattice of a K3 surface $X$ admitting a purely non-symplectic automorphism $\sigma=\sigma_n$ of order $n=14,21,28$ or $42$.  Under the assumption of generality we have:
\begin{equation}\label{eq-NS}r \doteq \rk NS(X)=22-d_{n}\cdot\varphi(n)\end{equation}
and using the results obtained in the previous sections we are able to describe $NS(X)$ in every case. 
Since the invariant lattices $S(\sigma^i)$ are all primitively embedded in $NS(X)$ by \cite[Section 3]{Nikulin}, if we can find one power $i$ such that the corresponding invariant lattice has the expected rank $r$, then we can conclude we have equality $S(\sigma^i)=NS(X)$. 

We will call a pair $(X,\sigma_n)$ satisfying \eqref{eq-NS} as above a general pair and we will use the classification of automorphisms of prime orders in \cite{ArtebaniSartiTaki} and \cite{Nikulin-inv} in order to describe explicitly the lattices $NS(X)$. 

Our results are presented in Propositions \ref{thmNS14} and \ref{thmNS} below: 

\begin{proposition}
Let $(X,\sigma_{14})$ be a general pair. For each possibility listed in Table \ref{tab:ex}, with the exception of case $C1(0,2)$ (see Remark \ref{c1(0,2)}), the N\'eron--Severi lattice $NS(X)$ is as in Table \ref{NS14} below.

\begin{table}[H]\centering
\begin{tabular}{c!{\vrule width 1.5pt}c|c|c|c|c}
&$\chi_{14}$&$\chi_7$&$\chi_2$&$(d_{14},d_7,d_2,d_1)$&$NS(X)$\\
\noalign{\hrule height 1.5pt}
A1&7&3&-14&(3,0,1,3)&$S(\sigma_7)=U\oplus K_7$\\
A1&7&3&0&(2,1,0,4)&$S(\sigma_2)$\\
\hline
A2&5&3&-16&(3,0,2,2) &$S(\sigma_7)=U\oplus K_7$\\

\noalign{\hrule height 1.5pt}
B3&14&10&0&(2,0,0,10)&$S(\sigma_7)=U\oplus E_8$\\
\noalign{\hrule height 1.5pt}
C1&6&10&-8&(2,0,4,6)&$S(\sigma_7)=U(7)\oplus E_8$\\
\hline
C2&4&10&-10&(2,0,5,5)&$S(\sigma_7)=U(7)\oplus E_8$\\
\hline
C3&8&10&-6&(2,0,3,7)&$S(\sigma_7)=U(7)\oplus E_8$\\
\noalign{\hrule height 1.5pt}
D2&3&17&-4&(1,0,8,8)&$S(\sigma_7)=U\oplus E_8\oplus A_6$\\\hline
D3&7&17&0&(1,0,6,10)&$S(\sigma_7)=U\oplus E_8\oplus A_6$\\\hline
D8&13&17&6&(1,0,3,13)&$S(\sigma_7)=U\oplus E_8\oplus A_6$\\
\end{tabular}
\caption{}
\label{NS14}
\end{table}
\label{thmNS14}
\end{proposition}

\begin{proof}
For $n=14$, one has $\varphi(14)=6$ and by Remark \ref{d14}, $\rk S(\sigma_7)=d_2+d_1,  \rk S(\sigma_2)=6d_7+d_1$.
By \eqref{eq-NS} one has
\begin{itemize}
    \item if $d_{14}=3$, $\rk NS(X)=4$: 
  \item if $d_{14}=2$, $\rk NS(X)=10$;
    \item if $d_{14}=1$, $\rk NS(X)=16$.
\end{itemize}
and we get that the Néron--Severi lattice $NS(X)$ is as in Table \ref{NS14}.
\end{proof}

\begin{remark}
If $\sigma=\sigma_{14}$ is a purely non-symplectic automorphism of order $14$ on a K3 surface $X$ such that $\sigma^2$ is of type $D$, then $NS(X)=S(\sigma^2)$. In fact, we know that $r=\rk NS(X)\geq 16=\rk S(\sigma^2)$, hence the rank $\ell$ of the transcendental lattice $NS(X)^{\perp}$ is at most $6$. But since $\ell$ must be divisible by $\varphi(14)=6$, it must be the case that $\ell=6$ and $r=16$.
\label{NScaseD}
\end{remark}

\begin{remark}
For a general pair $(X,\sigma_{14})$ such that $\Fix(\sigma_{14})$ is of type $C1(0,2)$, none of the invariant lattices $S(\sigma_{14}^i)$ have the expected rank. Thus we are not able to compute the N\'eron--Severi lattice of the general K3 surface in this case.
\label{c1(0,2)}

\end{remark}

When $n=21,28$ or $42$, we have that $\varphi(21)=\varphi(28)=\varphi(42)=12$ and for all cases $d_n=1$, thus $\rk NS(X)=22-12=10$. We prove:

\begin{proposition}
If $n=21,28$ or $42$, the description of the lattice $NS(X)$ for a general pair $(X,\sigma_n)$ is as follows:
\begin{enumerate}[(i)]
    \item{If $n=21$, the possibilities are shown in the following table:
    \vspace{\baselineskip}
\begin{center}
\begin{tabular} {c|c|c|c|c|c}
Type $\sigma_{21}$ & $\chi_{21}$ & $\chi_7$ & $\chi_3$ & $(d_{21},d_7,d_3,d_1)$&$NS(X)$ \\
\noalign{\hrule height 1.5pt}
   C(3,2,3)& 10 & 10 & 3 & (1,0,1,8)&$S(\sigma_7)=U(7)\oplus E_8$ \\ \hline
   C(3,1,2)& 7 & 10 & 0 & (1,0,2,6)&$S(\sigma_7)=U(7)\oplus E_8$\\ \hline
   C(3,0,1)& 4 & 10 & -3 & (1,0,3,4) &$S(\sigma_7)=U(7)\oplus E_8$\\ \hline
     B(3,3,4)&13&10 & 6&(1,0,0,10)&$S(\sigma_7)=U\oplus E_8$
\end{tabular}
\end{center}
\vspace{\baselineskip}}
\item{Similarly, if $n=28$ we have the following table of possibilities:
\vspace{\baselineskip}
\begin{center}
\begin{tabular} {c|c|c|c|c|c|c|c}
Type $\sigma_{14}$ & $\chi_2$ & $\chi_4$ & $\chi_7$ & $\chi_{14}$ & $\chi_{28}$ &  $(d_{28},d_{14},d_7,d_4,d_2,d_1)$ &$NS(X)$ \\
\noalign{\hrule height 1.5pt}
A1(3,2) & 0 & 12 & 3 & 7 & 5 & (1,0,1,0,0,4)&$S(\sigma_2)$\\ \hline
A1(3,2) & 0 &  -4 & 3 & 7 & 3 & (1,1,0,0,2,2)&$S(\sigma_2)$\\ \hline
B3(6,5) & 0  & 12 & 10 &  14 & 12 & (1,0,0,0,0,10)&$S(\sigma_2)$
\end{tabular}
\end{center}\vspace{\baselineskip}}
\item{And if $n=42$ we have:
\vspace{\baselineskip}
\begin{center}\begin{tabular} {c|c|c|c|c|c}
Type $\sigma_{14}$ & $\chi_{21}$ & $\chi_7$ & $\chi_3$ & $(d_{42},d_{21}, d_{14},d_7,d_6,d_3,d_2, d_1)$&$NS(X)$\\
\noalign{\hrule height 1.5pt}
  C1 & 10 & 10 & 3 &   (1,0,0,0,1,0,2,6)&$S(\sigma_7)=U(7)\oplus E_8$ \\ \hline
  C3 & 7 & 10 & 0 & (1,0,0,0,1,1,1,5) & $S(\sigma_7)=U(7)\oplus E_8$ \\ \hline
  B3 & 13 & 10 & 6& (1,0,0,0,0,0,0,10) &$S(\sigma_7)=U\oplus E_8$
\end{tabular}
\end{center}\vspace{\baselineskip}}
\end{enumerate}
\label{thmNS}
\end{proposition}

\begin{proof}
It follows from Remarks \ref{d21}, \ref{d28} and \ref{d42}.
\end{proof}

\begin{remark}
We observe that when $n=28$ and $\Fix(\sigma_{14})$ is of type $A1(3,2)$, then the 2-elementary lattice $S(\sigma_2)$ has invariants $(r,a)=(10,6)$. But, a priori, the invariant $\delta$ is not unique. By \cite[Theorem 0.1]{GS}, we have that $\delta=0$ if and only if $X$ also admits a symplectic involution.
\end{remark}

\bibliographystyle{plain}
\bibliography{biblio14}

\end{document}